\documentclass[11pt,letterpaper,oneside]{article}
\usepackage[left=1.5in,top=1.5in,right=1.5in,bottom=1.5in]{geometry}

\usepackage[utf8]{inputenc}
\usepackage[nottoc,numbib]{tocbibind}
\usepackage[english]{babel}
\usepackage[T1]{fontenc}
\usepackage{xcolor}
\usepackage{amsmath,amsthm,amssymb,amsfonts}
\usepackage[all,cmtip]{xy}
\usepackage{faktor}
\usepackage{verbatim}
\usepackage{graphicx}
\usepackage{csquotes}
\usepackage{float}

\usepackage{hyperref}
\hypersetup{
    colorlinks,
    linkcolor=darkred,
    citecolor=darkgreen,
    urlcolor=darkblue}
\definecolor{darkblue}{rgb}{0.0,0.0,0.5}
\definecolor{darkgreen}{rgb}{0.0,0.5,0.0}
\definecolor{darkred}{rgb}{0.8,0.0,0.0}

\newcommand{\del}{\partial}
\newcommand{\Cb}{\mathbb{C}}
\newcommand{\Qb}{\mathbb{Q}}
\newcommand{\Zb}{\mathbb{Z}}
\newcommand{\Rb}{\mathbb{R}}
\newcommand{\Pb}{\mathbb{P}}

\DeclareFontFamily{U}{rcjhbltx}{}
\DeclareFontShape{U}{rcjhbltx}{m}{n}{<->rcjhbltx}{}
\DeclareSymbolFont{hebrewletters}{U}{rcjhbltx}{m}{n}
\DeclareMathSymbol{\mem}{\mathord}{hebrewletters}{109}

\makeatletter
\renewcommand*\env@matrix[1][\arraystretch]{%
  \edef\arraystretch{#1}%
  \hskip -\arraycolsep
  \let\@ifnextchar\new@ifnextchar
  \array{*\c@MaxMatrixCols c}}
\makeatother

\let\svthefootnote\thefootnote
\newcommand\freefootnote[1]{%
  \let\thefootnote\relax%
  \footnotetext{#1}%
  \let\thefootnote\svthefootnote%
}

\theoremstyle{definition}
\newtheorem{theorem}{Theorem}[section]
\newtheorem{definition}[theorem]{Definition}
\newtheorem{example}[theorem]{Example}
\newtheorem{remark}[theorem]{Remark}
\newtheorem{proposition}[theorem]{Proposition}
\newtheorem{assumptions}[theorem]{Assumptions}

\newtheorem*{acknowledgements*}{Acknowledgements}

\date{}

\title{A-model Implications of Extended Mirror Symmetry}
\author{Lukas Hahn and Johannes Walcher}

\begin{document}

\freefootnote{\textit{Date}: \today }

\maketitle

\begin{abstract}
Associativity of the quantum product ensures flatness of the Dubrovin connection and is the basis
for Hodge-theoretic mirror symmetry of Calabi-Yau threefolds.  We use ring and module structure on
cohomology pertaining to a Lagrangian submanifold to define an extension of the A-model Variation of
mixed Hodge structure that matches the predictions from extended mirror symmetry.  Our construction
makes contact with axioms for open Gromov-Witten theory recently proposed by Solomon-Tukachinsky.
\end{abstract}

\section{Introduction}

Initiated by the predictions of \cite{candelas_pair_1991}, mirror symmetry has been a source of
plentiful novel mathematics related to the geometry of string compactifications.  In its original
form, it asserts a relationship between \enquote{mirror pairs} of Calabi-Yau threefolds $X$ and $Y$,
whose utilization as extra dimensions in string theory leads to the same effective physics in four
dimensions.  In terms of the topological phases of string theory, it can be understood as an
equivalence between the A-model on $X$ and the B-model on $Y$, implying the characteristic
relationship of Hodge numbers
\begin{equation}\label{mirroring}
h^{3-p,q}(X)=h^{p,q}(Y)
\end{equation} 
and the encoding of highly non-trivial enumerative information of $X$ in the Hodge theory of $Y$.
In particular, the Gromov-Witten invariants of $X$ can be extracted from the periods that determine
the variation of Hodge structure (VHS) associated to the middle dimensional cohomology of the mirror
family $\mathcal{Y}\rightarrow \Delta^*$.  This objective crucially involves the asymptotic behavior
of this VHS near a point of \enquote{maximal degeneracy} in the moduli space, that can be locally
described as a variation of mixed Hodge structure (VMHS) of a certain (Hodge-Tate) type
\cite{deligne_local_1997}.  A Hodge-theoretic framework for the mirror phenomenon is now based on
the observation that this object admits an A-model interpretation in terms of the $H^2$-module
structure on the even dimensional cohomology $H^{\text{even}}(X)$ defined by the quantum product
\cite{morrison_mirror_1993,morrison_mathematical_1997} .  As observed in
\cite{cattani_asymptotic_2001,cattani_frobenius_2003} the appearance of variations of MHS of this
type in the A-model can be viewed as a general consequence of a cup product deformed by a quantum
potential subject to a WDVV-type of equation.  An equivalence of the respective variations defines
the notion of Hodge-theoretic mirror pair and is the consequence of Mirror Theorems
\cite{givental_equivariant_1996,lian_mirror_1997} which imply the enumerative predictions of mirror
symmetry in case of the quintic.
\\
Open strings and D-branes enter the stage in the homological mirror symmetry program
\cite{kontsevich_homological_1994}, in which mirror pairs of Calabi-Yau threefolds are characterized
by an equivalence of the $A_\infty$ D-brane categories of the A- and B-model.  In the A-model, the
derived Fukaya category $\text{Fuk}(X)$ is defined via Floer theory, with objects given by suitably
decorated Lagrangian submanifolds of $X$, while on the B-side, the derived category of coherent
sheaves is purely classical. Homological mirror symmetry then amounts to the statement
\begin{equation}
\text{Fuk}(X)\cong D^b(Y),
\end{equation}
which is by now proven for the quintic \cite{sheridan_homological_2015} and known to imply
Hodge-theoretic closed string mirror symmetry \cite{ganatra_mirror_2015} .
\\
While the enumerative geometry of Lagrangian submanifolds implicitly appears in the definition of the
Fukaya category \cite{fukaya_counting_2011}, an explicit calculation of \enquote{open} Gromov-Witten
invariants according to the traditional mirror principle has been put forward in
\cite{walcher_opening_2007}. In analogy to the closed case, the invariants associated to a given
Lagrangian arise in an expansion around a point of maximal degeneracy of a Hodge-theoretic quantity
attached to the mirror object. Physically, this involves an equivalence between the corresponding A-
and B-brane superpotentials, which are the basic observables attached to a given brane vacuum.  On
the B-side, the superpotential is given by the holomorphic Chern-Simons invariant.  In many
situations, it can be calculated from an algebraic cycle $\mathcal{C}\rightarrow \Delta^*$ that
represents the algebraic Chern class of the B-brane. Namely, the algebraic cycle determines a Hodge
theoretic normal function that generally classifies extensions of variations of MHS
\cite{morrison_d-branes_2009}.
\\
Until now, this extension of the VMHS in the B-model has not been interpreted geometrically in the
A-model, partly due to the unavailability of sufficiently general axioms for open Gromov-Witten
invariants, which were only recently found in a series of papers by Solomon-Tukachinsky
\cite{solomon_differential_2020, solomon_point-like_2021, solomon_relative_2021}.  The purpose of
this paper is to fill this gap by providing an A-model description for the extended VMHS of the
B-model.
\\
The paper is organized as follows.  In Section \hyperref[section2]{2} we give a review of closed
string mirror symmetry formulated in a Hodge-theoretic language, following
\cite{doran_algebraic_2014, da_silva_jr_arithmetic_2016}.
This includes the degenerating behavior of
the VHS in the B-model and the reconstruction of the resulting objects based on A-model data.  While
focusing on the one-parameter case, we emphasize the role of the WDVV equations in a multi-parameter
situation and conclude with a description of Hodge-theoretic mirror pairs and the standard example
of \cite{candelas_pair_1991}. Section \hyperref[section3]{3} is concerned with the extension of
B-model data by algebraic cycles and the associated normal functions following
\cite{morrison_d-branes_2009,schwarz_integrality_2017}.  We put a special emphasis on the asymptotic
behavior of the normal functions around a point of maximal degeneracy and derive the general
structure of the B-brane superpotential based on monodromy considerations.  An A-model
interpretation of this superpotential is given in Section \hyperref[section4]{4}.  We begin with a
characterization of A-brane superpotentials, taking into account the Solomon-Tukachinsky axioms for
open Gromov-Witten invariants, and define the notion of extended mirror pair.  A partial recipe to
reconstruct the relevant normal function in the A-model based on an extension of the quantum product
by holomorphic disks is presented.  We conclude with the formulation of an extended Mirror Theorem
for the main example of \cite{walcher_opening_2007}.

\begin{acknowledgements*}
The authors thank A.\ Alcolado, S.\ Nill and S.\ Tukachinsky for various fruitful discussions on
extensions. J.W.\ is grateful to H.\ Jockers and D.\ Morrison for on-going collaboration on the
limiting value of the D-brane superpotenial. The work of L.H.\ is supported by a fellowship from
the Evangelisches Studienwerk Villigst.  This work is also supported in part by the Deutsche
Forschungsgemeinschaft (DFG, German Research Foundation) under Germany's Excellence Strategy EXC
2181/1 - 390900948 (the Heidelberg STRUCTURES Excellence Cluster).
\end{acknowledgements*}

\section{Closed String Mirror Symmetry}\label{section2}

\subsection{Maximal Degenerations in the B-model}

Let $\pi:\mathcal{Y}\rightarrow \Delta^*$ be a smooth family of projective Calabi-Yau threefolds  over a punctured disk $\Delta^*$ with local coordinate $z$, for which each fiber $\pi^{-1}(z)=Y_z$ is simply connected and has Hodge numbers $h^{3,0}=h^{2,1}=h^{1,2}=h^{0,3}=1$.
We think of the base as a small neighborhood around a boundary point of the compactified moduli space $\overline{\mathcal{M}}(Y)$ of complex structures on a reference fiber $Y$ and further assume that there is a semi-stable continuation to the disk $\Delta$
\begin{equation}
\xymatrix{
\mathcal{Y}\,\ar@{^{(}->}[r]\ar[d] & \overline{\mathcal{Y}}\ar[d] & \,Y_0 \ar@{_{(}->}[l] \ar[d]\\
\Delta^*\,\ar@{^{(}->}[r] & \Delta & \,\{0\}, \ar@{_{(}->}[l]
}
\end{equation}
where the singular fiber $Y_0$ has at most normal crossing singularities. 
As $Y_z$ varies across $\Delta^*$, its middle-dimensional integral cohomology fits into a $\Zb$-local system given by the higher direct image $\mathcal{H}^3_\Zb=R^3\pi_*\Zb$ with fibers $(\mathcal{H}^3_\Zb)_z=H^3(Y_z;\Zb)$.
Over $\Cb$, these cohomology groups admit a Hodge decomposition
\begin{equation}
	H^3:=H^3(Y_z;\Cb)=H^{3,0}(Y_z)\oplus H^{2,1}(Y_z)\oplus H^{1,2}(Y_z)\oplus H^{0,3}(Y_z)
\end{equation}
for which there is a decreasing Hodge filtration
\begin{equation}
	F^p=\bigoplus_{i\geq p}H^{i,3-i}(Y_z)
\end{equation}
with $F^3\subset F^2\subset F^1\subset F^0=H^3$.
The Hodge filtration varies holomorphically across $\Delta^*$, such that $\mathcal{F}^p:=F^p\otimes \mathcal{O}_{\Delta^*}$ defines a corresponding filtration of holomorphic subbundles on $\mathcal{H}^3:=\mathcal{H}^3_\Zb\otimes \mathcal{O}_{\Delta^*}$.
A natural antisymmetric polarization form $Q(\cdot,\cdot):\mathcal{H}_\Zb^3\otimes \mathcal{H}_\Zb^3\rightarrow \Zb$ is induced from the Poincar\'e duality pairing on integral cohomology and can be extended linearly to $\mathcal{H}^3$.
The local system $\mathcal{H}^3_\Cb=\mathcal{H}^3_\Zb\otimes \Cb$ uniquely determines the flat Gau\ss-Manin connection $\nabla: \mathcal{H}^3\rightarrow \mathcal{H}^3\otimes \Omega^1_{\Delta^*}$ with $\nabla(\mathcal{H}^3_\Cb)=0$ that satisfies the Griffiths transversality condition $\nabla(\mathcal{F}^p)\subset \mathcal{F}^{p-1}\otimes \Omega^1_{\Delta^*}$.
This collection of algebraic data $(\mathcal{H}_\Zb^3,\mathcal{H}^3,\mathcal{F}^\bullet,\nabla,Q)$ defines a polarized, integral VHS of weight 3 over $\Delta^*$.\\

For a small loop $\gamma(t)$ around the singular point $0\in\Delta$ based at some $z\in \Delta^*$, we can lift the class $g_z\in H^3(Y_z;\Zb)$ to a flat section $g(t)\in H^3(Y_{\gamma(t)};\Zb)$ over $[0,1]$ with $g(0)=g$.
The monodromy operator $M:\mathcal{H}^3_\Zb\rightarrow \mathcal{H}^3_\Zb$ is defined by $M(g)=g(1)$ in each fiber and guaranteed to be quasi-unipotent by the monodromy theorem.
Assuming that $M$ is unipotent, the monodromy logarithm $N:=\log(M):\mathcal{H}^3_\Qb\rightarrow \mathcal{H}^3_\Qb$ is a nilpotent operator on $\mathcal{H}^3_\Qb=\mathcal{H}^3_\Zb\otimes \Qb$ which induces the unique increasing monodromy weight filtration $W_\bullet :=W_\bullet(N)$ of the form $W_{-1}=\{0\}\subset W_0\subset \cdots \subset W_6=\mathcal{H}^3_\Qb$.
It satisfies the properties 
\begin{equation}\label{monodromyprop}
	N(W_k)\subset W_{k-2},\qquad N^k:\text{Gr}_{3+k}^W\overset{\sim}{\longrightarrow} \text{Gr}_{3-k}^W,
\end{equation}
where $\text{Gr}_{k}^W=W_k/W_{k-1}$ are the graded pieces for $k=0,...,3$.
The monodromy logarithm can be employed to define the untwisted local system with corresponding connection
\begin{equation}\label{untwisting}
	\widetilde{\mathcal{H}}^3_\Zb:=\exp\left(-\frac{\log(z)}{2\pi i }N\right)\mathcal{H}^3_\Zb,\qquad \nabla^c:=\nabla+\frac{N}{2\pi i} \,\frac{dz}{z},
\end{equation}
from which Deligne's canonical continutation $\widetilde{\mathcal{H}}^3:=\widetilde{\mathcal{H}}^3_\Zb\otimes \mathcal{O}_{\Delta}\rightarrow \Delta$ over the puncture can be constructed. 
As a consequence of the nilpotent orbit theorem \cite{schmid_variation_1973}, the Hodge filtration also extends to holomorphic subbundles $\widetilde{\mathcal{F}}^p\subset\widetilde{\mathcal{H}}^3=\widetilde{\mathcal{H}}^3_\Zb\otimes\mathcal{O}_\Delta$, with limiting filtration $F^\bullet_0\subset \widetilde{\mathcal{H}}^3_0$.
This canonical extension of the Hodge-theoretic data allows to assign a meaningful limit to the VHS at the puncture that carries information about the degeneration of the fiber $Y_0$ over $0\in \Delta$.
For this purpose, we untwist a multivalued basis $\{g_i\}$ of $\mathcal{H}^3_\Zb$ via \eqref{untwisting} and consider the local system $\widetilde{\mathcal{H}}^3{}_{\Zb,0}$ generated by $\widetilde{g}_i(0)$.
The collection of data $(\widetilde{\mathcal{H}}{}^3_{\Zb,0},\widetilde{\mathcal{H}}{}^3_0,F^\bullet_0,W_\bullet)$ defines a MHS called the \textit{limiting mixed Hodge structure} (LMHS) of $\mathcal{H}^3$.\\

In the given situation, the LMHS can be classified in terms of the monodromy behavior around the puncture \cite{green_neron_2007}, where of particular interest in the context of mirror symmetry is the case of a \textit{maximal degeneration}, characterized by maximally unipotent monodromy (MUM) with $(M-id)^4=0$ and $(M-id)^3\neq0$ at $z=0$. 
Here, the LMHS is Hodge-Tate of the form
\begin{equation}\label{LMHS}
\xymatrix{
	\Zb(-3) \ar[r]^-N & \Zb(-2) \ar[r]^-N & \Zb(-1) \ar[r]^-N & \Zb(0)
}
\end{equation}
with $W_{2i}=\text{ker}\,N^{i+1}$ and $\text{Gr}^W_{2i}\cong \Zb(-i)$ for $i=0,...,3$ while otherwise $\{0\}$. 
In a neighborhood around the puncture, the limiting Hodge filtration can be extended to a filtration that is constant with respect to \eqref{untwisting}, leading to a VHS
\begin{equation}
	\mathcal{H}^3_\text{nilp}:=(\mathcal{H}^3_\Zb,\mathcal{H}^3,\exp(-\frac{\log(z)}{2 \pi i}N)\,F^\bullet_0,\nabla,Q),
\end{equation}
called the \textit{nilpotent orbit}, which approximates the original VHS asymptotically.
Griffiths transversality follows from $N(F^p_0)\subset F^{p-1}_0$ and together with the weight filtration $\mathcal{W}_\bullet=W_\bullet \otimes \mathcal{O}_{\Delta^*}$ the nilpotent orbit defines a Hodge-Tate VMHS on $\Delta^*$ with LMHS \eqref{LMHS}.
Indeed, following Deligne \cite{deligne_local_1997}, the filtration $\mathcal{W}_\bullet$ pairs up also with the original Hodge filtration such that the data $(\mathcal{H}^3_\Zb,\mathcal{H}^3,\mathcal{F}^\bullet,\mathcal{W}_\bullet,\nabla,Q)$ defines a polarized, integral VMHS of Hodge-Tate type in a neighborhood $\Delta^*$ of the MUM-point.
Its periods consist of a $\nabla^c$-constant part encoded in the nilpotent orbit, as well as holomorphic components vanishing at $0\in \Delta$ \cite{cattani_eduardo_degenerating_1989}.\\

We denote by $\Omega\in\Gamma(\mathcal{F}^3,\Delta^*)$ a choice of nonvanishing holomorphic 3-form on $\mathcal{Y}$.
For any multivalued flat section $g$ of $\mathcal{H}^3_\Cb$ the periods $Q(g,\Omega)$ are holomorphic functions on $\Delta^*$ satisfying a Picard-Fuchs differential equation of generalized hypergeometric type \cite{doran_mirror_2006}.
In a neighborhood of the MUM-point, it can be expressed in terms of a canonical logarithmic vector field via
\begin{equation}
	D_{\text{PF}}(\theta)\,Q(g,\Omega)=0, \qquad \theta=2\pi i\,\frac{d}{d\log(q)}.
\end{equation}
While the periods generally diverge when approaching the puncture, the \textit{limiting period} $Q(\widetilde{g}(0),\Omega)$ defined in terms of the untwisted local system is well defined and determines the LMHS at $z=0$.\\

We can find integral flat generators $g_i\in W_{2i}\cap \mathcal{H}^3_\Zb$ and a basis $\{e_j\}:=\{e_3,e_2,e_1,e_0\}$ with $e_j\in\mathcal{F}^j$, which lead to a structure that can be reinterpreted from an A-model point of view.
Starting with the holomorphic 3-form $e_3=\Omega$ and the generator $g_0$ Poincar\'e dual to the minimal integral vanishing cycle, the function $\varpi_0=Q(g_0,e_3)$ is up to a constant the unique fundamental period which is analytic at $z=0$.
Assuming the monodromy is \textit{small} \cite{morrison_mirror_1993}, there exists a further generator $g_1$ with period $\varpi_1=Q(g_1,e_3)$ that transforms according to $M(g_1)=g_1+g_0$ such that 
\begin{equation}\label{mirrormap}
	q(z):=\exp(2\pi i t),\qquad t=\frac{\varpi_1}{\varpi_0},
\end{equation}
is a well-defined function that serves as the unique canonical coordinate on $\Delta^*$.
Following \cite{schwarz_integrality_2017}, we can complete $\{g_0,g_1\}$ to an integral flat basis $\{g_i\}:=\{g_3,g_2,g_1,g_0\}$ with periods $\varpi_i=Q(g_i,e_3)$ in which the full monodromy and its logarithm are represented by the matrices
\begin{equation}\label{monodromyandlogarithm}
M=
\begin{pmatrix}
1 & 0 & 0 & 0\\
-1 & 1 & 0 & 0\\
0 & \kappa & 1 & 0\\
-\frac{a+2\kappa}{12} & \kappa & 1 & 1
\end{pmatrix},\qquad
N=
\begin{pmatrix}
0 & 0 & 0 & 0\\
-1 & 0 & 0 & 0\\
\frac{\kappa}{2} & \kappa & 0 & 0\\
-\frac{a}{12} & \frac{\kappa}{2} & 1 & 0
\end{pmatrix},\qquad \kappa,a\in \Zb,
\end{equation}
and the polarization form is given by 
\begin{equation}\label{polarizationform}
Q=
\begin{pmatrix}
 0 & 0 & 0 & 1\\
 0 & 0 & 1 & 0\\
 0 & -1 & 0 & 0\\
 -1 & 0 & 0 & 0
\end{pmatrix}.
\end{equation}
Choosing \eqref{mirrormap} as the canonical coordinate and normalizing $e_3$ such that $\varpi_0(z)=1$, the flat basis $\{g_i\}$ is related to the Hodge basis $\{e_j\}$ by
\begin{equation}\label{canonicalbasis}
	\begin{split}
&g_0=e_0,\quad
g_1=e_1+ \log q\,e_0,\quad
g_2=e_2+ \theta^2\mathcal{F} \,e_1+ \theta\mathcal{F} \,e_0\\
g_3&=e_3- \log q\,e_2+\left(\theta\mathcal{F}-\log q \,\theta^2 \mathcal{F} \right) e_1+\left(2\mathcal{F}-\log q\,\theta \mathcal{F}\right)e_0,
\end{split}
\end{equation}
in terms of the prepotential
\begin{equation}\label{prepotential}
	\mathcal{F}=\frac{1}{(2\pi i)^3}\frac{\kappa}{6}\log(q)^3+\frac{1}{(2\pi i)^2}\frac{\kappa}{4}\log(q)^2-\frac{1}{(2\pi i)}\frac{a}{24}\log(q)+f(q).
\end{equation}
Here, $f(q)$ is a single valued function, which extends holomorphically over the puncture
\begin{equation}\label{Bmodelholomorphic}
	f(q)=\frac{\widetilde{b}\,\zeta(3)}{(2\pi i)^3}+\frac{1}{(2\pi i)^3}\sum_{d=1}^\infty \widetilde{N}_d\,q^d,\qquad \widetilde{b}\in \Qb.
\end{equation}
When the family $\mathcal{Y}$ is defined over $\Qb$, the coefficients $\widetilde{N}_d\in\Qb$ of the $q$-series expansion are generally rational numbers.

\begin{remark}\label{fourloop}
The constant term of \eqref{Bmodelholomorphic} determines the limit of the period $\varpi_3$
\begin{equation}
	\lim_{q\rightarrow 0}\varpi_3:=Q(\widetilde{g}_3(0),e_3)=\frac{b\,\zeta(3)}{(2\pi i)^3},\quad b=2\widetilde{b},
\end{equation}
in the LMHS and corresponds to an extension class in $\text{Ext}^1_{\text{MHS}}(\Zb(-3),\Zb(0))=\Cb/\Zb(3)$, whose theoretical origin is explained in \cite{green_neron_2010}. From a physical perspective, it can be interpreted as a four-loop correction to the sigma model metric on the corresponding A-model Calabi-Yau background \cite{candelas_pair_1991,grisaru_four-loop_1986}.
\end{remark}

With the logarithmic vector field $\theta$, the Gau\ss-Manin connection $\nabla(\theta):=\nabla_t$ in the basis $\{e_j\}$ is fully determined by the prepotential in terms of the \textit{Yukawa coupling} $\mathfrak{C}$,
\begin{equation}
\nabla_t
= d + 
\begin{pmatrix}
	0 & 0 & 0 & 0\\
	1 & 0 & 0 & 0\\
	0 & -\mathfrak{C} & 0 & 0 \\
	0 & 0 & -1 & 0
\end{pmatrix}
\otimes \frac{dq}{2\pi i\,q},\qquad 
\mathfrak{C}=\theta^3\mathcal{F}=Q(\nabla_t^3\,e_3,e_3),
\end{equation}
which has the role of the three-point correlation function in the closed B-model.

\subsection{A-model Variations of Hodge Structure}

Closed string mirror symmetry can be established by a reconstruction of these Hodge-theoretic structures from objects that are inherent to the associated A-model geometry \cite{morrison_mathematical_1997}.
Let $X$ be a simply connected and projective Calabi-Yau threefold with Hodge numbers $h^{i,i}=1$ for $i=0,...,3$. On the even dimensional cohomology 
\begin{equation}
	H^{\text{even}}:=H^{\text{even}}(X;\Cb)=H^{0,0}(X)\oplus H^{1,1}(X)\oplus H^{2,2}(X)\oplus H^{3,3}(X)
\end{equation}
we define the decreasing A-model Hodge filtration 
\begin{equation}
	F^p=\bigoplus_{i\leq 3-p}H^{i,i}(X)\subset H^{\text{even}}
\end{equation}
with $F^3\subset F^2\subset F^1\subset F^0=H^{\text{even}}$, inspired by the characteristic symmetry of Hodge numbers \eqref{mirroring}.
For $t\in \mathfrak{H}$ in the upper half plane, we denote by $\omega=B+iJ=t[H]$ the complexified K\"ahler class and consider a region of the K\"ahler moduli space $\overline{\mathcal{K}}(X)$ given by a punctured disk $\Delta^*$ parametrized by $q=e^{2\pi i t}$ around a large radius limit point $q=0$.
In order to construct an object that mirrors the VMHS coming from $\mathcal{Y}$ we define the vector bundle $\mathcal{H}^{\text{even}}:=H^{\text{even}}\otimes \mathcal{O}_{\Delta^*}$ to which we can extend the filtration by $\mathcal{F}^p:=F^p\otimes \mathcal{O}_{\Delta^*}$. 
The antisymmetric polarization form $Q(\cdot,\cdot):\mathcal{H}^\text{even}\otimes \mathcal{H}^\text{even}\rightarrow \mathbb{C}$ is again induced from the cup-product on $H^{\text{even}}$ and given by
\begin{equation}
Q(\alpha, \beta) =
(-1)^i\int_X\alpha\cup \beta ,\quad \alpha\in H^{i,i}(X),\quad \beta\in H^{3-i,3-i}(X).
\end{equation}
In the basis $\{e_j\}:=\{e_3,e_2,e_1,e_0\}$ given by
\begin{equation}
e_3=[X],\quad e_2=[H],\quad e_1=-[\ell],\quad e_0=[p]
\end{equation}
the polarization form is represented by the matrix \eqref{polarizationform} with entries $Q_{ij}=Q(e_i,e_j)$ and coefficients of the inverse matrix denoted by $Q^{ij}=(Q_{ij})^{-1}$.\\

We denote by $GW_\beta(\alpha_1,\ldots,\alpha_n)$ the Gromov-Witten invariant of degree $\beta\in H_2(X;\Zb)$ with $n$ interior constraints given by cycles Poincar\'e dual to $\alpha_k\in H^2(X)$ and consider the Gromov-Witten potential
\begin{equation}
\Phi=\frac{1}{(2\pi i )^3}\sum_{\beta\in H_2(X;\Zb)}GW_\beta\, q^\beta=\frac{1}{6}\int_X\omega\cup \omega\cup \omega+\Phi_h=\frac{\kappa}{6}t^3+\Phi_h.
\end{equation}
Here, $\kappa=\int_XH\cup H\cup H$ is the classical triple intersection number and 
\begin{equation}\label{GWpot}
\Phi_h=	\frac{1}{(2\pi i )^3}\sum_{\beta\in H_2(X;\Zb)\setminus\{0\}}GW_\beta\, q^\beta,\qquad q^\beta=\exp(2\pi i \int_\beta \omega),
\end{equation}
is the quantum part of the potential, which we will assume to converge and think of as a holomorphic function on $\Delta^*$.
The Dubrovin connection on $\mathcal{H}^{\text{even}}$ is defined in terms of the small quantum product $e_2*(-):H^{\text{even}}(X)\rightarrow H^{\text{even}}(X)$ with the hyperplane class $e_2=[H]$, i.e. 
\begin{equation}
	\nabla_t(e_k)=e_2*e_k=\sum_{l,m}\sum_{\beta\in H_2(X;\Zb)}GW_\beta(e_2,e_k,e_l)\,q^\beta\,Q^{lm}\,e_m,
\end{equation}
whose energy zero contribution $e_2*e_k|_{q=0}=e_2\cup e_k\in H^{\text{even}}(X)$ corresponds to the ordinary cup product.
On the bundle $\mathcal{H}^{\text{even}}$ with Hodge basis $\{e_j\}$ the action of the Dubrovin connection is given by
\begin{equation}
\nabla_t :=d+\Big([H]*\Big)\otimes dt
= d + 
\begin{pmatrix}
	0 & 0 & 0 & 0\\
	1 & 0 & 0 & 0\\
	0 & -\Phi^{\prime\prime\prime} & 0 & 0 \\
	0 & 0 & -1 & 0
\end{pmatrix}
\otimes dt.
\end{equation}
It records the variation of the small quantum product with respect to the K\"ahler parameter and the analog of the Yukawa coupling is here identified as the closed A-model three-point correlation function
\begin{equation}
	\Phi^{\prime\prime\prime}=\sum_{\beta}GW_\beta(H,H,H)\,q^\beta=\int_XH*H*H=Q(\nabla_t^3e_3,e_3),
\end{equation}
by means of the Kontsevich-Manin axioms \cite{kontsevich_gromov-witten_1994}. 
Any $\alpha\in H^{2k}(X)$ for $k\leq n-p$ can be regarded as a section of $\mathcal{F}^p$ and the Griffiths transversality condition $\nabla(\mathcal{F}^p)\subset \mathcal{F}^{p-1}\otimes \Omega^1_{\Delta*}$ follows from the fact that $[H]*\alpha\in H^{2k+2}(X)$ is a section of $\mathcal{F}^{p-1}$. 
In the one-parameter case, flatness of $\nabla$ is an immediate consequence of associativity and commutativity of the small quantum product.
In the notation of Doran-Kerr \cite{doran_algebraic_2014}, the complex local system $\mathcal{H}^{\text{even}}_\Cb=\text{ker}(\nabla)$ can be systematically constructed in terms of $\Phi$ and the basis $\{e_j\}$ by setting 
\begin{equation}\label{Clocalsystem}
\begin{split}
	\widetilde{\sigma}(e_0):=e_0,\quad &\widetilde{\sigma}(e_1):=e_1,\quad \widetilde{\sigma}(e_2):=e_2+\Phi_h^{\prime\prime}\, e_1+\Phi_h^{\prime}\, e_0,\\
	&\widetilde{\sigma}(e_3):=e_3+\Phi_h^\prime\, e_1+2\Phi_h\, e_0
\end{split}
\end{equation}
and then defining the quantum deformed classes by $g_i=\sigma(e_i):=\widetilde{\sigma}(e^{-\omega}\cup e_i)$. This leads to a flat basis which corresponds to a complex solution of the quantum differential equation defined by the A-model connection \cite{cox_mirror_1999}.

\begin{remark}\label{WDVV}
In the one-parameter setting, flatness of the Dubrovin connection is an automatic consequence of the algebraic properties of the operation $[H]*(-)$.
In the case $h^{1,1}(X)=n$ of several Kähler parameters, we denote by
\begin{equation}
	\omega=\sum_{i=1}^{n}t_i\,[H_i]\longmapsto q=(q_1,\ldots,q_n)=\left(e^{2\pi i t_1},\ldots ,e^{2\pi i t_n}\right)\in \left(\Delta^*\right)^n,\quad t_i\in \mathfrak{H},
\end{equation}
the Kähler class and corresponding coordinates of the Kähler moduli space around a small punctured polydisk around the large radius limit point.
The A-model curvature is now given by
\begin{equation}
R_\nabla(\del_i,\del_j)\big([H_k]\big)=[H_i]*\big([H_j]*[H_k]\big)-[H_j]*\big([H_i]*[H_k]\big),
\end{equation}
and vanishes by commutativity and associativity of the small quantum product.
While commutativity is a given on $H^{\text{even}}$, associativity leads to constraints on the Gromov-Witten invariants which can be formulated as a system of partial differential equations satisfied by the Gromov-Witten potential
\begin{equation}\label{WDVVequations}
	\sum_{a,b}\del_a\del_i\del_j\,\Phi \cdot Q^{ab}\cdot  \del_b\del_k\del_l\,\Phi=\sum_{a,b}\del_a\del_i\del_k\,\Phi \cdot Q^{ab}\cdot  \del_b\del_j\del_l\,\Phi,\quad \text{for all } i,j,k,l.
\end{equation}
\vspace{-0.5cm}
\begin{figure}[H]
	\centering
  \includegraphics[width=0.75\textwidth]{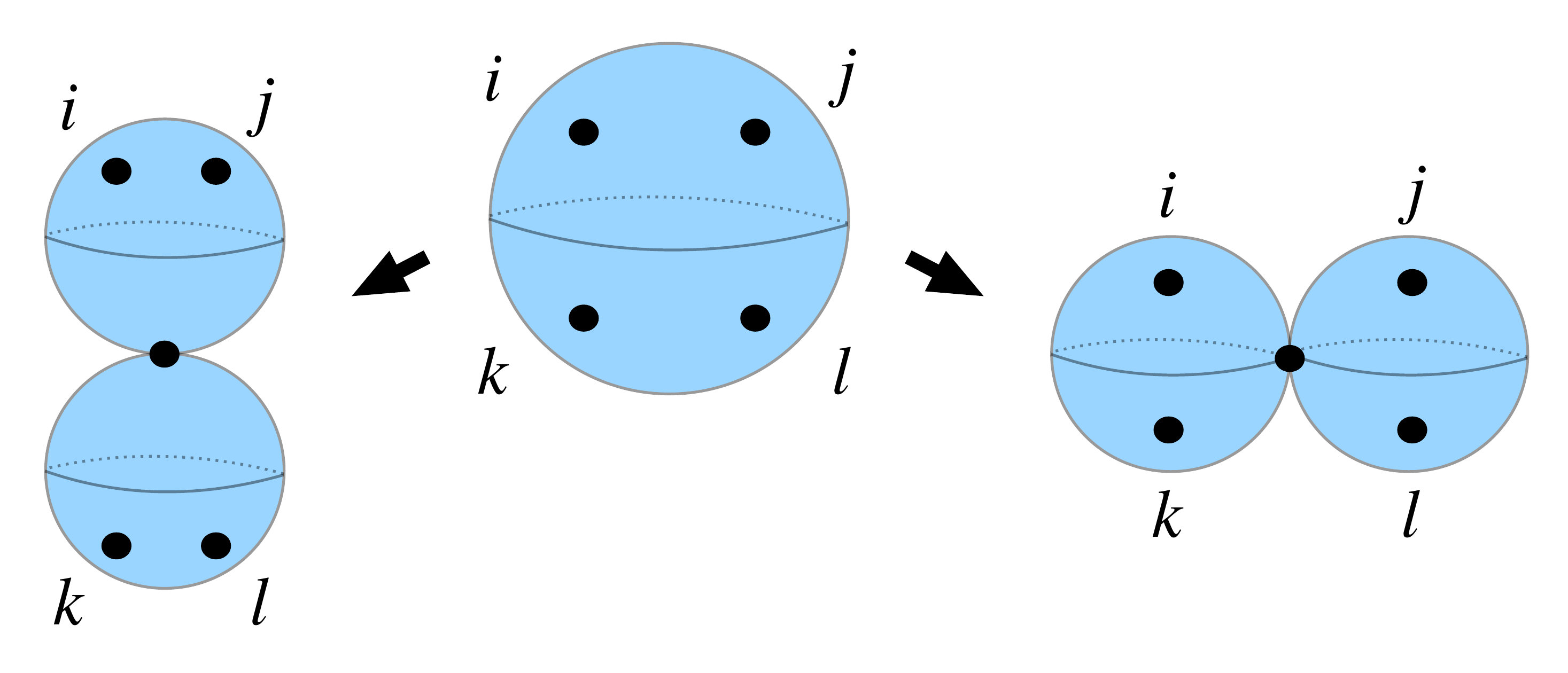}
\end{figure}
These are known as the Witten-Dijkgraaf-Verlinde-Verlinde (WDVV) equations and can be understood as an equivalence of Gromov-Witten invariants associated to different boundary strata of the moduli space of stable maps with four marked points \cite{witten_structure_1990, dijkgraaf_notes_1991,dijkgraaf_topological_1991}. 
\end{remark}

In order to study the asymptotics of the A-model periods $Q(g_i,e_j)$ we turn again to the monodromy $M:\mathcal{H}^{\text{even}}_\Cb\rightarrow \mathcal{H}^{\text{even}}_\Cb $ around the puncture.
Its logarithm at $q=0$ can be computed from the residue $\text{Res}_{q=0}(\nabla)$, which is seen to coincide with the energy zero contribution to the quantum product \cite{cox_mirror_1999}
\begin{equation}
	\text{Res}_{q=0}(\nabla)=\frac{1}{2\pi i}\left(\sum_{l,m}GW(e_j,e_k,e_l)\,Q^{lm}\,e_m\right)=\frac{1}{2\pi i}\left(\sum_{l,m}\left(\int_X e_j\cup e_k\cup e_l\right)\,Q^{lm}\,e_m\right).
\end{equation}
The monodromy logarithm $N_{q=0}=-2\pi i\,\text{Res}_{q=0}(\nabla)$ is therefore given by the cup product $-[H]\cup(-): H^{\text{even}}(X)\rightarrow  H^{\text{even}}(X)$, which in the basis $\{e_j\}$ is represented by the matrix
\begin{equation}
N_{q=0}=
\begin{pmatrix}
0 & 0 & 0 & 0\\
-1 & 0 & 0 & 0\\
0 & \kappa & 0 & 0\\
0 & 0 & 1 & 0	
\end{pmatrix}.
\end{equation}
The weight filtration $W_\bullet:=W(N)_\bullet$ associated to this monodromy is at $0\in \Delta$ given by
\begin{equation}
W_k=\bigoplus_{i\geq 3-k/2}H^{2i},
\end{equation}
invoking the Hard Lefschetz theorem $[H]^k\cup(-):H^{3-k}\xrightarrow{\sim} H^{3+k}$.
We regard the graded pieces $\text{Gr}_k^WH^{\text{even}}=H^{2n-k}$ as having a pure Hodge structure of weight $k$ and denote by $F^\bullet_0=\widetilde{\mathcal{F}}^\bullet_0$ the limit of the trivially extended Hodge filtration $\widetilde{\mathcal{F}}^\bullet=F^\bullet\otimes \mathcal{O}_\Delta$.  
By untwisting the local system according to \eqref{untwisting}, this defines a MHS $(\widetilde{\mathcal{H}}^{\text{even}}_{\Zb,0},\widetilde{\mathcal{H}}_0^\text{even},F^\bullet_0,W_\bullet)$ of Hodge-Tate type at $0\in \Delta$.
Similar to the construction in the B-model, we can extend it to a nilpotent orbit 
\begin{equation}
	\mathcal{H}^{\text{even}}_{\text{nilp}}:=(\mathcal{H}^{\text{even}}_\Cb,\mathcal{H}^\text{even},e^{-tN}F^\bullet_0,\nabla,Q)
\end{equation}
in a small neighborhood around the puncture, with Griffiths transversality following from $N(F_0^p)\subset F_0^{p-1}$ and LMHS as above.
Paired up with the filtration $\mathcal{W}_\bullet=W_\bullet \otimes \mathcal{O}_{\Delta^*}$, this defines a Hodge-Tate VMHS that is constant with respect to the corresponding untwisted connection \eqref{untwisting} and fully determined by the $H^2$-module structure on $H^{\text{even}}$ induced by the cup product.
Again, also the original Hodge filtration leads to a polarized, complex VMHS $(\mathcal{H}^{\text{even}}_\Cb,\mathcal{H}^\text{even},\mathcal{F}^\bullet,\mathcal{W}_\bullet,\nabla,Q)$ of Hodge-Tate type on $ \Delta^*$, in which the additional holomorphic components of the periods are derived from the quantum Gromov-Witten potential \eqref{GWpot}.
This can be viewed as a general consequence of the WDVV equations for a deformation of this $H^2$-module structure to a quantum product defining a \textit{Frobenius module} in the language of Cattani-Fernandez \cite{cattani_asymptotic_2001,cattani_frobenius_2003}.\\

So far the construction of this VMHS in the A-model only involved the specification of a complex local system as a solution to the quantum differential equation defined by the quantum connection. 
A subtle problem arises in the definition of the correct integral local system, which is reflected by the fact that the limiting A-model periods $Q(\widetilde{g}_i,e_j)$ do not reproduce the limiting period matrix of the B-model, including the crucial zeta value $\zeta(3)$.
A direct construction of this integral local system $\mathcal{H}^{\text{even}}_\Zb$ from a pure A-model perspective is due to Iritani \cite{iritani_integral_2009} and Katzarkov-Kontsevich-Pantev \cite{katzarkov_hodge_2008}.
It starts with a basis $\{\xi_i\}$ of the algebraic K-theory $K^{\text{alg}}_0(X)$ which is symplectic with respect to the Mukai pairing
\begin{equation}
	\langle \xi, \xi^\prime\rangle_{M} =\int_X ch(\xi^\vee \otimes \xi^\prime)\cup Td(X),
\end{equation}
where $Td(X)$ denotes the Todd class of $X$.
From this, the integral quantum deformed classes $g_i:=\sigma\left(\hat{\Gamma}(X)\cup ch(\xi_i)\right)$ are defined in terms of a \enquote{square-root} of $Td(X)$, called the Gamma class
\begin{equation}
	\hat{\Gamma}(X):=\exp\left(\sum_{k\geq 2}\frac{(-1)^k(k-1)!}{(2\pi i )^k}\,\zeta(k)\,ch_k(X)\right)\in H^{\text{even}}
\end{equation}
of $X$. 
For a Calabi-Yau threefold with Chern classes $c_1=0$, $c_2=a\,[\ell]$ and $c_3=b\,[p]$, the Gamma class becomes 
\begin{equation}\label{gammaclass}
	\hat{\Gamma}(X)=1-\frac{\zeta(2)}{(2\pi i)^2}\,c_2+\frac{\zeta(3)}{(2\pi i)^3}\left(c_1c_2-c_3\right)=[X]+\frac{a}{24}[\ell]-\frac{b\,\zeta(3)}{(2\pi i)^3}[p]
\end{equation}
and the right choice of K-theory basis reproduces the flat basis \eqref{canonicalbasis} in the B-model, with prepotential \eqref{prepotential} \cite{da_silva_jr_arithmetic_2016}.
The large radius monodromy and its logarithm in the basis $\{g_i\}:=\{g_3,g_2,g_1,g_0\}$ are given by \eqref{monodromyandlogarithm} and the corresponding limiting periods correctly reproduce the asymptotic behavior of the B-model, where the previously missing values are from an A-model perspective provided by the Gamma class. 
The collection of data $(\mathcal{H}^{\text{even}}_\Zb,\mathcal{H}^\text{even},\mathcal{F}^\bullet,\mathcal{W}_\bullet,\nabla,Q)$ now defines an integral, polarized \textit{A-model VMHS} on $\Delta^*$ associated to $(X,\omega)$.

\subsection{Hodge-Theoretic Mirror Pairs}\label{mirrorpairs}

The two variations of MHS associated to A-model and B-model geometries around suitably chosen limit points in the respective moduli spaces can be used to formulate closed string mirror symmetry in a Hodge-theoretic framework:
A pair of Calabi-Yau threefold families $(X,\omega)$ and $\mathcal{Y}$ is called a \textit{(Hodge-theoretic) mirror pair} if the polarized variations of MHS on $\mathcal{H}^{\text{even}}(X)$ and $\mathcal{H}^3(\mathcal{Y})$ are isomorphic
 \begin{equation}\label{mirrorpair}
\xymatrix{
  \mathcal{H}^{\text{even}} \ar[d] \ar[r]^-{\sim} & \mathcal{H}^3 \ar[d]   \\
  \Delta^*_q \ar[r]^-{m} & \Delta^*_z,
}
\end{equation}
where $\Delta^*_q$ and $\Delta^*_z$ are small neighborhoods around MUM-type boundary points of $\overline{\mathcal{K}}(X)$ and $\overline{\mathcal{M}}(Y)$ related by the \textit{mirror map} $m:q(z)\mapsto z(q)$.
Given a mirror pair, the periods of $\Omega$ in the B-model attain an A-model interpretation as quantum corrected volumes of the integral cycles $g_i$, e.g.
\begin{equation}\label{correctedvolume}
	\theta\mathcal{F}=\frac{\kappa}{2}t^2+\frac{\kappa}{2}t-\frac{a}{24}+\frac{1}{(2\pi i)^2}\sum_{d=1}^\infty d\widetilde{N}_d\,q^d=Q(g_2,e_3),
\end{equation}
and the choice of integral local system corresponds to specifying their asymptotics in terms of the limiting period matrix
\begin{equation}\label{limitingperiods}
	\Pi_{\,q=0}:=Q(\widetilde{g}_i(0),e_j)_{i,j=0,\ldots,3}=
\begin{pmatrix}[1.3]
	1 & 0 & 0 & 0 \\ 
	0 & 1 & 0 & 0 \\
	\frac{a}{24} & - \frac{\kappa}{2} & 1 & 0 \\
	\frac{b\,\zeta(3)}{(2\pi i)^2} & \frac{a}{24} & 0 & 1
\end{pmatrix}
\end{equation}
which determines the LMHS at $q=0$.
The quantum differential equations coming from the Gau\ss-Manin and Dubrovin connections are identified, leading to an equality of the respective A-model and B-model three-point functions under the mirror map.
From this, the Yukawa coupling of $\mathcal{Y}$
\begin{equation}
	\mathfrak{C}(q(z))  =\kappa+\sum_{d=1}^\infty d^3\widetilde{N}_d\,q^d=\Phi^{\prime\prime\prime}(q),
\end{equation}
computes the Gromov-Witten invariants of the mirror $X$ with $\widetilde{N}_d=GW_\beta$ for $\beta=d[\ell]$.
While in general $\widetilde{N}_d\in \Qb$ for families defined over the rationals, an integral expansion of the holomorphic part of the Gromov-Witten potential is given by a multicover formula
\begin{equation}\label{multicoverformula}
	\Phi_h=\frac{1}{(2\pi i)^3}\sum_{d=1}^\infty \widetilde{N}_d\,q^d=\frac{1}{(2\pi i)^3}\sum_{d=1}^\infty N_d\,\text{Li}_3(q^d):=\frac{1}{(2\pi i)^3}\sum_{d,k=1}^\infty \widetilde{N}_d\,\frac{q^{dk}}{k^3},\quad N_d\in \Zb,
\end{equation}
where the instanton numbers $N_d$ are closely related to the number of degree $d$ curves lying in $X$.

\begin{example}\label{closedexample}

The most well-known example of a mirror pair was orignally studied in \cite{candelas_pair_1991} and consists of the Fermat quintic
\begin{equation}\label{fermatquintic}
		X=\left\{x_1^5+x_2^5+x_3^5+x_4^5+x_5^5=0\right\}\subset \Pb^4
\end{equation} 
on the A-side together with the mirror quintic family $\mathcal{Y}$ on the B-side, given by a smooth resolution $\widetilde{Y_z/G}$ of 
\begin{equation}
	Y_z=\left\{W=x_1^5+x_2^5+x_3^5+x_4^5+x_5^5-5\psi\, x_1x_2x_3x_4x_5 =0 \right\}\subset \Pb^4,\quad z=(5\psi)^{-5},
\end{equation}
with the Greene-Plesser group $G\cong(\Zb/5\Zb)^5$ of symmetries leaving $Y_z$ invariant.
The statement of the Mirror Theorem \cite{lian_mirror_1997} for the quintic is that $(X,\omega)$ and $\mathcal{Y}$ form a mirror pair.
This implies that the Picard-Fuchs operator of $\mathcal{Y}$ \cite{doran_mirror_2006}
\begin{equation}
	D_{\text{PF}}\,(-)=(2\pi i)^2\left[\theta^4-5z\prod_{k=1}^4\left(5\theta+k\right)\right]\,(-)
\end{equation}
associated to the standard choice of holomorphic 3-form 
\begin{equation}\label{standardform}
	\Omega_z=\left(\frac{5}{2\pi i}\right)^3\text{Res}_{\,W=0}\frac{\sum_{i=1}^5(-1)^ix_i\,dx_1\wedge \cdots \wedge \widehat{dx_i}\wedge\cdots \wedge dx_5}{W}
\end{equation}
leads to a Yukawa coupling
\begin{equation}
	\mathfrak{C}=5+2875\, q+4876875\,q^2+8564575000\,q^3+\cdots 
\end{equation}
computing the Gromov-Witten invariants of $X$ with 
\begin{equation}
	(2\pi i)^3\,\Phi_h  =2875\,q+\frac{4876875}{8} \,q^2 +  \frac{8564575000}{27}  \,q^3 + \cdots .
\end{equation}
In terms of the limiting periods of $\Omega$ the LMHS is specified by by
\begin{equation}
	Q(\widetilde{g}_3(0),e_2)=\frac{25}{12},\qquad Q(\widetilde{g}_3(0),e_3)=-\frac{25i\,\zeta(3)}{\pi^3},
\end{equation}
where the constants $a=50$ and $b=-200$ of \eqref{limitingperiods} arise in the A-model as constituents of the Gamma class \eqref{gammaclass}.
\end{example}

\section{B-Branes and Normal Functions}\label{section3}

\subsection{Extensions from Algebraic Cycles}
We consider for each fiber $Y_z$ of the family $\pi:\mathcal{Y}\rightarrow \Delta^*$ an algebraic curve $i:C_{z}\hookrightarrow Y_z$ varying with $z\in \Delta^*$ such that $\pi\circ i:\mathcal{C}\rightarrow \Delta^*$ is a smooth family which admits a semi-stable continuation over $\Delta$.
When the family of cycles $\mathcal{C}\in \text{CH}^2(\mathcal{Y})_{\text{hom}}$ is homologically trivial, meaning that $i_*([C_z])=0\in H_2(Y_z;\Zb)$ for each $z\in \Delta^*$, we can associate to the pair $(\mathcal{Y},\mathcal{C})$ a VMHS that can be realized as an extension
\begin{equation}\label{extension}
0 \longrightarrow \mathcal{H}^3\longrightarrow \hat{\mathcal{H}}{}^3 \longrightarrow  \mathcal{I}\longrightarrow  0,\qquad \hat{\mathcal{H}}{}^3\in \text{Ext}_{\text{VMHS}}(\mathcal{I},\mathcal{H}^3),
\end{equation}
of the pure VHS $\mathcal{I}$ of weight 4 by $\mathcal{H}^3$.
The extension is fiberwise coming from the exact sequence
\begin{equation}
	0 \longrightarrow  H^3(Y_z;\Cb) \longrightarrow  H^3(Y_z\setminus C_z;\Cb) \longrightarrow \text{ker}\left(H^0(C_z;\Cb)\overset{i_!}{\longrightarrow} H^4(Y_z;\Cb)\right) \longrightarrow 0,
\end{equation}
where $i_!:=PD\circ i_*\circ PD$ denotes the Gysin map.
The extension $0\rightarrow \mathcal{H}^3_\Zb\rightarrow \hat{\mathcal{H}}{}^3_\Zb\rightarrow \mathcal{I}_\Zb\rightarrow 0$  of local systems arises geometrically from the short exact sequence
\begin{equation}
0 \rightarrow  H_3(Y_z;\Zb) \rightarrow  H_3(Y_z,C_z;\Zb) \rightarrow \text{ker}\big(H_2(C_z;\Zb)\rightarrow H_2(Y_z;\Zb)\big) \rightarrow 0\\
\end{equation}
by Poincar\'e duality with
\begin{equation}
	\text{ker}\big(H_2(C_z;\Zb)\rightarrow H_2(Y_z,\Zb)\big)\cong H^0(C_z;\Zb),
\end{equation}
after a suitable Tate-twist such that $\mathcal{I}_{\Zb,z}=\Zb(-2)$ is Hodge-Tate of pure Hodge type (2,2).
The extended Hodge filtration $\hat{F}^p$ on the cohomology of the complement is coming from the grading of the sheaf $\Omega^\bullet_{\mathcal{Y}}\langle \log \mathcal{Z}\rangle$ and satisfies
\begin{equation}
	\faktor{\hat{F}^2}{F^2}=\Zb(-2)^d,\qquad \faktor{\hat{F}^p}{\hat{F}^{p+1}}= \faktor{F^p}{F^{p+1}}.
\end{equation}
We define the filtration $\hat{\mathcal{F}}^p=\hat{F}^p\otimes \mathcal{O}_{\Delta^*}$ of holomorphic subbundles on $\hat{\mathcal{H}}{}^3=\hat{\mathcal{H}}{}^3_\Zb\otimes \mathcal{O}_{\Delta^*}$ and the Gau\ss-Manin connection canonically extends to $\hat{\mathcal{H}}{}^3$ by specifying the horizontal sections according to $\nabla(\hat{\mathcal{H}}{}^3_\Cb)=0$.
The only non-vanishing components of the weight filtration $\mathcal{W}_\bullet$ on $\hat{\mathcal{H}}{}^3_\Qb$ are set to $\mathcal{W}_3=\mathcal{H}^3_\Qb$ and $\mathcal{W}_4=\hat{\mathcal{H}}^3_\Qb$, such that the data  $(\hat{\mathcal{H}}{}^3_\Zb,\hat{\mathcal{H}}{}^3,\hat{\mathcal{F}}^\bullet,\mathcal{W}_\bullet,\nabla,Q)$ defines an integral, polarized VMHS that fits as an extension in the exact sequence \eqref{extension}.\\

By a theorem of Carlson \cite{carlson_extensions_1979}, extensions of variations of HS of this kind are classified by the intermediate Jacobian fibration
\begin{equation}
\text{Ext}_{\text{VMHS}}^1(\mathcal{I}_\Zb,\mathcal{H}^3)\cong \mathcal{J}^2(\mathcal{H}^3):=\frac{\mathcal{H}^3}{\mathcal{F}^2+\mathcal{H}^3_\Zb},
\end{equation}
in terms of a normal function $\nu$ of $\mathcal{H}^3$ given by a holomorphic section of $\mathcal{J}^2(\mathcal{H}^3)$, satisfying the horizontality condition $\nabla \widetilde{\nu}\in \mathcal{F}^1\otimes \Omega_{\Delta^*}$, where $\widetilde{\nu}$ is any lift of $\nu$ to $\mathcal{H}^3$.
The section which determines a given extension can be explicitly described by taking both an integral lift $h$ and an $\hat{\mathcal{F}}^2$-lift $f$ of the generator of $\mathcal{I}_\Zb$ in $\hat{\mathcal{H}}{}^3$, such that the normal function is defined as the class of the difference $[h-f]\in\mathcal{J}^2(\mathcal{H}^3)$.
Because $\nabla^2=0$ and modulo the ambiguity of the $\hat{\mathcal{F}}^2$-lift $f$, the derivatives $\nabla\widetilde{\nu}$ characterize the $\Cb$-VMHS locally by the class
\begin{equation}\label{infinvariant}
	\big[\nabla\widetilde{\nu}\big]\in\frac{\text{ker}\big(\nabla:\mathcal{F}^1\otimes \Omega^1_{\Delta^*}\rightarrow \mathcal{F}^0\otimes \Omega^2_{\Delta^*}\big)}{\text{im}\big(\nabla:\mathcal{F}^2\rightarrow \mathcal{F}^1\otimes \Omega^1_{\Delta^*}\big)},
\end{equation}
called \textit{Griffiths' infinitesimal invariant} of $\mathcal{H}^3$ \cite{voisin_hodge_2002}.
Its vanishing implies that the normal function is locally constant, meaning that locally there always exists a lift $\widetilde{\nu}\in \mathcal{H}^3$ satisfying $\nabla\widetilde{\nu}=0$.\\

In the geometric situation of a homologically trivial algebraic cycle,
Poincar\'e duality implies an isomorphism
\begin{equation}\label{normalfunction}
	\mathcal{J}^2(\mathcal{H}^3)\cong \faktor{\left(\mathcal{F}^2\right)^\vee}{\left(\mathcal{H}^3_\Zb\right)^\vee},
\end{equation}
which allows to characterize the extension by a functional $\nu_{\mathcal{C}}:\mathcal{F}^2\rightarrow \mathcal{O}_{\Delta^*}$ defined modulo integral periods.
To define its action in each fiber $Y_z$, we choose any relative 3-chain $\Gamma_{z}\in H_3(Y_z,C_{z};\Zb)$ with boundary $\del \Gamma_{z}=C_{z}$ and define the integral lift as a current
\begin{equation}
	h=(2\pi i)^2\delta_{\Gamma_{z}}\in H^3\big(Y_z\setminus C_{z} ;\Zb(2)\big)\cong H_3\big(Y_z,C_{z};\Zb(-1)\big).
\end{equation}
The $\hat{\mathcal{F}}^2$-lift is given by another current $f\in \hat{F}^2$ with $df=(2\pi i)^2\delta_{C_{z}}$.
Their difference $h-f$ is a class in $H^3(Y_z)$, well-defined modulo $F^2H^3(Y_z)+H^3(Y_z;\Zb(2))$ and the equivalence class in the intermediate Jacobian $[h-f]\in J^2(H^3)=\mathcal{J}^2(\mathcal{H}^3)_z$ gives rise to a functional defined up to integral periods by
\begin{equation}
	\nu_{\mathcal{C}}(\eta)_z:=\int_{Y_z} \left(h -f\right)\wedge \eta =(2\pi i)^2\int_{\Gamma_{z}} \eta,\qquad \eta\in F^2H^3(Y_z),
\end{equation}
where the second term vanishes by type considerations.
The \textit{Abel-Jacobi map}
\begin{equation}
	AJ:\text{CH}^2(\mathcal{Y})_{\text{hom}} \longrightarrow \mathcal{J}^2(\mathcal{H}^3),\qquad 
	\mathcal{C} \longmapsto \nu_{\mathcal{C}},
\end{equation}
is defined by assigning this functional to the associated homologically trivial algebraic cycle.
Vanishing of the infinitesimal invariant $[\nabla\widetilde{\nu}_\mathcal{C}]=0$ is then geometrically reflected by the cycle $\mathcal{C}$ being locally constant, i.e. independent of the modular parameter $z$.
The truncated normal function
\begin{equation}
\mathcal{W}_{B}(z)=Q(AJ(\mathcal{C}),e_3)=(2\pi i)^2\int_{\Gamma_{z}} \Omega_z
\end{equation}
can be physically interpreted as the B-brane superpotential associated to the open string vacuum determined by the cycle $\mathcal{C}$.
It satisfies an inhomogeneous version of the Picard-Fuchs equation \cite{morrison_d-branes_2009}
\begin{equation}
	D_{\text{PF}}\,\mathcal{W}_B=\mathcal{J}
\end{equation}
with an inhomogeneity that is additive with respect to the corresponding cycles.\\

Even when $\mathcal{C}$ is not homologically trivial itself, we can define a version of the Abel-Jacobi map depending on the choice of a homologically equivalent cycle $\mathcal{C}_0$ that serves as a fixed reference with $\mathcal{C}-\mathcal{C}_0\in \text{CH}^2(\mathcal{Y})_\text{hom}$.
The normal function is then fiberwise determined by integrating $\eta\in F^2H^3(Y_z)$ over a relative 3-chain $\Gamma_z$ with boundary $\del\Gamma_z=C_z-C_{0,z}$.
We will therefore generally allow the cycle to be reducible, i.e. $C_z=\bigcup_{k}C_{z,k}$ with $k=0,...,d$ and $\mathcal{C}=\bigcup_k\mathcal{C}_k$, assuming that all irreducible components are homologically equivalent to each other, i.e. $[C_{z,k}]-[C_{z,k^\prime}]=0\in H_2(Y_z,\Zb)$, for each pair of indices $k$ and $k^\prime$.
We choose $\mathcal{C}_0\subset \mathcal{C}$ as fixed reference cycle with respect to which homological equivalence is measured and denote by $AJ(\mathcal{C}_k)=\nu_{\mathcal{C}_k}$ the normal function described above, slightly abusing notation.
This determines an extension similar to \eqref{extension}, where $\mathcal{I}_{\Zb,z}=\Zb(-2)^{\oplus d}$.
We interpret the truncated normal functions $\mathcal{W}_k$ as the B-brane superpotential associated to the open string vacuum determined by $\mathcal{C}_k$ relative to the reference $\mathcal{C}_0$.
For any two homologically equivalent cycles $\mathcal{C}_k,\mathcal{C}_{k^\prime}\subset\mathcal{C}$ the differences 
\begin{equation}
	\mathcal{T}_{k,k^\prime}:=\mathcal{W}_k-\mathcal{W}_{k^\prime},\qquad D_{\text{PF}}\,\mathcal{T}_{k,k^\prime}=\mathcal{J}_k-\mathcal{J}_{k^\prime},
\end{equation}
have an interpretation as BPS domain wall tensions between the open string vacua associated to $\mathcal{C}_k$ and $\mathcal{C}_0$.

\subsection{Asymptotics of the Abel-Jacobi Map}\label{AJasymptotics}

When the extended monodromy $\hat{M}:\hat{\mathcal{H}}^3_\Zb\rightarrow\hat{\mathcal{H}}^3_\Zb$ is unipotent, its logarithm $\hat{N}=\log(\hat{M})$ is a nilpotent operator that preserves the filtration $\mathcal{W}_\bullet$ of $\hat{\mathcal{H}}^3$.
In such a situation there exists a relative weight filtration $\hat{\mathcal{W}}_\bullet:=\hat{\mathcal{W}}_\bullet(\hat{N},\mathcal{W}_\bullet)$ on $\hat{\mathcal{H}}^3_\Qb$ such that $\hat{N}(\hat{\mathcal{W}}_k)\subset \hat{\mathcal{W}}_{k-2}$ and $\hat{\mathcal{W}}_\bullet$ induces on each $\text{Gr}^\mathcal{W}_k$ the weight filtration of the map induced by $\hat{N}$.
With respect to this relative filtration, $\mathcal{I}_\Zb$ is generally of weight 4 and the extension can be described in terms of the canonical coordinate around a MUM-point and the Hodge basis $\{e_j\}$.
For any lift $\widetilde{\nu}\in \mathcal{H}^3$ the validity of $Q(\nabla_t \widetilde{\nu},e_3)=0$ is here independent of the $\mathcal{F}^2$-part of the $\hat{\mathcal{F}}^2$-lift $f$, such that maximal degeneracy provides a canonical lift of the normal function given by 
\begin{equation}\label{canonical}
	\widetilde{\nu}=h-f= \mathcal{V}\, e_1 + \mathcal{W}\, e_0, 
\end{equation} 
where $\mathcal{V}$ and $\mathcal{W}$ are the truncated normal functions of $\widetilde{\nu}$.
Both functions are holomorphic on $\Delta^*$ and the horizontality condition corresponds to the vanishing of the coefficient of $e_0$ in
\begin{equation}
\nabla_t\widetilde{\nu}=\nabla_t(h-f)=\theta \mathcal{V}\,e_1+\left(-\mathcal{V}+\theta\mathcal{W}\right)e_0,
\end{equation}
which implies $\mathcal{V}=\theta\mathcal{W}$.
The integral lift becomes part of the extended basis $\{g_i,h\}$ of $\hat{\mathcal{H}}{}^3_\Zb$ and satisfies by definition $\nabla(h)=0$, which further leads to
\begin{equation}\label{derivativeextgenerator}
\nabla_t f= -\theta \mathcal{V}\,e_1=-\theta^2\mathcal{W}\,e_1,
\end{equation}
such that the extension of the Gau\ss-Manin connection is fully determined by the function $\mathcal{W}$.
The open string analog of the Yukawa coupling is therefore given by a certain representative of Griffiths' infinitesimal invariant \eqref{infinvariant} in the canonical coordinate.
Namely, the derivative of the canonical lift $\nabla\widetilde{\nu}$ is fully determined by its $\mathcal{F}^1$ component, which motivates the definition \cite{schwarz_integrality_2017} 
\begin{equation}\label{infinitesimalinvariant}
	\mathfrak{D}:=-Q(\nabla_t\widetilde{\nu},\nabla_t e_3)=Q(\nabla_t^2\widetilde{\nu},e_3)=\theta^2\mathcal{W},
\end{equation}
where the second equality follows from differentiating the horizontality condition.
In particular, the full information of the extension can be recovered from the superpotential and its inhomogeneous Picard-Fuchs equation in the given situation.\\

In case of a collection of cycles $\mathcal{C}=\bigcup_{k}\mathcal{C}_{k}$ with fixed reference $\mathcal{C}_0$ described in the previous section, the monodromy is not guaranteed to be unipotent, as the cycle might branch when encircling the puncture, exchanging the connected components of $\mathcal{C}$.
This implies that $\hat{M}$ has finite order $r$ when restricted to $\mathcal{I}_\Zb$, where we assume for simplicity that all arising orbits are of the same order.
To ensure the existence of a globally well defined cycle, we have to pass to an $r$-fold cover $\hat{z}=z^{1/r}\mapsto z$, branched at $0\in \Delta$, and consider the pullback
\begin{equation}\label{cover}
\xymatrix{
\mathcal{C}\, \ar@{^{(}->}[r]\ar@/_0.5pc/[rd]  & \hat{\mathcal{Y}} \ar[r]\ar[d]  & \mathcal{Y} \ar[d] \\
& \Delta_{\hat{z}}^*  \ar[r] & \Delta_z^*.
}
\end{equation}
On the cover, the correct extended monodromy logarithm is then given by $\hat{N}= \log(\hat{M}^r)$, and after pullback along \eqref{cover}, the degree of the local system $\hat{\mathcal{I}}_\Zb$ is $d/r$.
The extended monodromy logarithm is now generally of the form \cite{green_neron_2010}
\begin{equation}
\hat{N}=
\begin{pmatrix}
	rN & L\\
	0 & 0 
\end{pmatrix},
\end{equation}
with $L\in \text{Hom}_\Zb(\hat{\mathcal{I}}_\Zb,\mathcal{H}^3_\Zb)$. 
While we can find a splitting of the sequence \eqref{extension} over $\Qb$, such that $L=0$ and $\hat{\mathcal{I}}_\Qb$ has weight 3, the integral local system $\hat{\mathcal{I}}_\Zb$ has weight 4 in general.
In the integral basis $\{g_i,h_k\}$, the logarithm $\hat{N}$ is therefore represented by 
\begin{equation}\label{extendedmonodromy}
\hat{N}=
\begin{pmatrix}
0 & 0 & 0 & 0 & 0 & \cdots & 0\\
-r & 0 & 0 & 0 & 0 & \cdots & 0\\
\frac{r\kappa}{2} & r\kappa & 0 & 0 & \lambda_1 & \cdots & \lambda_d\\
-\frac{ra}{12} & \frac{r\kappa}{2} & r & 0 & s_1 & \cdots & s_d\\
 0 & 0 & 0 & 0 & 0 & \cdots & 0\\
 \vdots & \vdots & \vdots & \vdots & \vdots & \ddots & \vdots \\
 0 & 0 & 0 & 0 & 0 & \cdots & 0
\end{pmatrix}, 
\end{equation}
where the extending generator $h_k$ is mapped to $\hat{N}(h_k)=\lambda_k\,g_1+s_k \, g_0\in \mathcal{W}_2$.
Over $\Qb$ we can find a basis such that $\hat{N}(h_k)=s_k\, g_0\in \mathcal{W}_0$ and $h_k\in\hat{\mathcal{W}}_3$, but this change of coordinates cannot be performed while retaining integrality of the generators in general \cite{schwarz_integrality_2017}.
The graded pieces of the filtration $\hat{\mathcal{W}}_\bullet$ then fit into the diagram (cf. \cite{usui_studies_2014})
\begin{equation}\label{Bweightfiltration}
	\xymatrix{
	& \text{Gr}^{\mathcal{W}}_6 \ar[d]_-{rN} \ar[r]^-{\sim} & \text{Gr}^{\hat{\mathcal{W}}}_6  \ar[d]^-{\hat{N}} & & \\
	0 \ar[r] & \text{Gr}^{\mathcal{W}}_4 \ar[d]_-{rN} \ar[r] & \text{Gr}^{\hat{\mathcal{W}}}_4  \ar[d]^-{\hat{N}} \ar[r] & \hat{\mathcal{I}}_\Zb \ar[r] \ar[d]^-{L} & 0\\
	0 \ar[r] & \text{Gr}^{\mathcal{W}}_2 \ar[d]_-{rN} \ar[r] & \text{Gr}^{\hat{\mathcal{W}}}_2 \ar[d]^-{\hat{N}} \ar[r] &  H  \ar[r] & 0\\
	& \text{Gr}^{\mathcal{W}}_0 \ar[r]^-{\sim} & \text{Gr}^{\hat{\mathcal{W}}}_0 , & &
	}
\end{equation}
where horizontal lines are exact sequences and $H$ corresponds to an $r$-torsion group given by the quotient
\begin{equation}
  H:= \faktor{\Zb\langle g_0\rangle }{r\Zb\langle g_0\rangle},
\end{equation}
reflecting the branching behavior of the cycle.
The extended Gau\ss-Manin connection in the basis $\{e_j,f_k\}$ is now given by 
\begin{equation}
\nabla_t 
= d + 
\begin{pmatrix}
	0 & 0 & 0 & 0 & 0 & \cdots & 0\\
	1 & 0 & 0 & 0 & 0 & \cdots & 0\\
	0 & -\mathfrak{C} & 0 & 0 & -\mathfrak{D}_1 & \cdots & -\mathfrak{D}_d \\
	0 & 0 & -1 & 0 & 0 & \cdots & 0\\
	0 & 0 & 0 & 0 & 0 & \cdots & 0\\
	\vdots & \vdots & \vdots & \vdots & \vdots & \ddots & \vdots \\
 0 & 0 & 0 & 0 & 0 & \cdots & 0
\end{pmatrix}
\otimes \frac{dq}{2\pi i\,q},\quad 
\mathfrak{D}_k=\theta^2\mathcal{W}_k=Q(\nabla_t^2\widetilde{\nu},e_3)
\end{equation}
and determined by the Yukawa coupling $\mathfrak{C}$ together with the infinitesimal invariants $\mathfrak{D}_k$ coming from all superpotentials.\\

On the $r$-fold cover we pass to the local coordinate $\hat{q}=q^{1/r}$ and the structure of the monodromy $\hat{N}$ implies that the superpotential is generally of the form
\begin{equation}\label{torsionsuperpotential}
\mathcal{W}_k = \frac{1}{(2\pi i)^2}\frac{\lambda_k}{r^2}\,\log(q)^2+\frac{1}{(2\pi i)}\frac{s_k}{r}\, \log(q) + w_k(q),\qquad \lambda_k\in\Zb,s_k\in H
\end{equation}
with a single valued function
\begin{equation}\label{specialvalue}
	w_k(q)=c_k+\frac{1}{(2\pi i)^2}\sum_{d=1}^\infty \widetilde{n}_d\,q^{d/r}
\end{equation}
extending holomorphically over the puncture.
Here, the constant $c_k$ is related to special values of $L$-functions associated to the algebraic number field $K/\Qb$ over which the cycle is defined \cite{green_neron_2010}, see also \cite{laporte_monodromy_2012,jefferson_monodromy_2014} for examples.
Even for families $\mathcal{Y}$ over $\Qb$, the coefficients $\widetilde{n}_d\in K/\Qb$ of the $q$-series expansion generally lie in this number field and, analogous to the situation of Remark \hyperref[fourloop]{2.1}, the constant term of the expansion is here related to the limit of the Abel-Jacobi map associated to $\mathcal{C}$ in the canonical coordinate.
According to the prescription in \cite{green_neron_2010}, this limit takes values in the \textit{N\'eron-Model}
\begin{equation}\label{neronmodel}
	AJ(C_{\hat{q}}) \xrightarrow{\hat{q}\rightarrow 0}  AJ(C_0)\in G\ltimes\faktor{\Cb}{\Zb(2)},
\end{equation}
with a finite group $G$ generated by cycles limiting to distinct components in the singular fiber $Y_0$ and $\text{Ext}^1_{\text{MHS}}(\Zb(-2),\Zb(0))\cong\Cb/\Zb(2)$ arising as the image of a regulator map acting on the degenerated cycle $C_0$. 
In case of the extension defined by the superpotential \eqref{torsionsuperpotential}, the finite group $G$ contains an $r$-torsion subgroup coming from the covering, and the components of the Abel-Jacobi limit
\begin{equation}\label{AJlimit}
	Q(\widetilde{h}(0),e_2)=\frac{s_k}{r},\qquad
	Q(\widetilde{h}(0),e_3)=c_k,
\end{equation}
can be expressed in terms of the untwisted local system which determines the LMHS of the extension.

\section{Extended A-model Variations of Hodge Structure}\label{section4}

\subsection{Open Gromov-Witten Theory and A-Brane Superpotentials}\label{section4.1}

In a neighborhood of a large radius limit point, the D-brane geometry of the A-model is encoded in the Fukaya category of $X$.
The relevant objects in the present context are BPS branes given by compact, special Lagrangian submanifolds $i:L\hookrightarrow X$ with respect to the Kähler class and holomorphic 3-form of $X$.
Morphisms between these objects are related to Lagrangian intersection Floer cohomology and the $A_\infty$-algebra associated to Lagrangian submanifolds \cite{fukaya_lagrangian_2009}.
The Floer cochain complex of $L$ can be thought of as a modification of the de Rham complex $\Omega^\bullet(L)$ in which the differential 
\begin{equation}
	\mathfrak{m}_1=d+\mathcal{O}\left(\exp(2\pi i\int_\beta \omega)\right), 
\end{equation}
is deformed by holomorphic disks of degree $\beta\in H_2(X,L;\Zb)$ with boundary on $L$.
This differential fits into a hierarchy of $A_\infty$ structure maps $\mathfrak{m}_k:\Omega^\bullet(L){}^{\otimes k}\rightarrow\Omega^\bullet(L)[2-k]$ related to the compactified moduli space $\widehat{\mathcal{M}}_{k+1}(\beta)$ of open stable maps $(D,\del D)\rightarrow (X,L)$ with $k+1$ marked points on the boundary.
The condition $\mathfrak{m}_1^2=0$ for $\mathfrak{m}_1$ to define a differential is generally obstructed when there is a non-zero \enquote{tadpole} anomaly $\mathfrak{m}_0\neq 0$ coming from the disk amplitude with one boundary marked point.
In case of Calabi-Yau threefolds the correction for this anomaly amounts to deforming the $A_\infty$ structure by a \textit{bounding cochain} $b\in \Omega^1(L)$ that satisfies the $A_\infty$-Maurer-Cartan equation
\begin{equation}
	\mathfrak{m}_0+\mathfrak{m}_1(b)+\mathfrak{m}_2(b,b)+\mathfrak{m}_3(b,b,b)+\cdots=0.
\end{equation}
The space of bounding cochains is then given by the critical locus of the function \cite{fukaya_counting_2011}
\begin{equation}\label{FOOOsuperpotential}
	\Psi(b)=\sum_{k}\frac{1}{k+1}\left\langle \mathfrak{m}_k(b,\ldots,b),b\right\rangle +\mathfrak{m}_{-1},\qquad \mathfrak{m}_{-1}=\sum_{\beta\in H_2(X,L;\Zb)}\mathfrak{m}_{-1,\beta}\,q^\beta,
\end{equation}
where the constant term $\mathfrak{m}_{-1}$ corresponds to a count of disks without boundary constraints.
This count is well defined for Calabi-Yau threefolds as the associated moduli space has dimension $\dim\mathcal{M}_{0}(\beta)=0$ and its inclusion makes $\Psi$ a numerical invariant on each critical point.
In \cite{solomon_differential_2020,solomon_point-like_2021,solomon_relative_2021}, Solomon-Tukachinsky further generalized this numerical invariant using the bulk-deformed $A_\infty$-algebra associated to Lagrangian submanifolds \cite{fukaya_counting_2011,fukaya_lagrangian_2009}, which allows to make contact with the closed string sector and define open Gromov-Witten invariants with both boundary and bulk constraints in great generality. 
Here, bulk constraints are specified by intersections of the interior of the disks with cycles in $X$, while boundary constraints are always point-like and correspond to the top-degree part of the bounding cochain.

\begin{remark}\label{remarkbubble}
When a holomorphic disk without boundary constraints has spherical degree $\beta=i_*\widetilde{\beta}\in \text{im}\left(H_2(X)\rightarrow H_2(X,L)\right)$, it can degenerate in a way in which the boundary collapses to a point in $L$.
These contributions spoil invariance of $\Psi$ but can be corrected when $L$ is homologically trivial by further including sphere invariants with one puncture lying on the 4-chain $\Gamma$ with $\del \Gamma=L$  \cite{pandharipande_disk_2008}.
A formalization of this procedure results in \textit{enhanced} open Gromov-Witten invariants \cite[Proposition 4.19]{solomon_relative_2021} which in general depend on the choice of $\Gamma$. In addition to the bounding cochain, it becomes part of the data making the disk count well defined.
\end{remark}

\begin{assumptions}
\label{assumpt}
Drawing from the known examples of Lagrangian submanifolds of compact Calabi-Yau threefolds, we assume that $L=\{x\in X\,|\, \iota(x)=x\}$ is the fixed point set of an anti-holomorphic involution $\iota:X\rightarrow X$, i.e. $\iota^*\omega=-\omega$ and $\iota^*\Omega=-\Omega$.
As a consequence, both $\omega$ and $\text{Re}(\Omega)$ restrict to zero on $L$, which therefore becomes a special Lagrangian submanifold of $X$.  
We further assume that $L$ is a rational homology sphere, i.e.
\begin{equation}
	H_k(L;\Qb)=H_k(S^3,\Qb)\quad \text{for all } k.
\end{equation}
We will allow $r$-torsion in $H_1(L;\Zb)\cong H^2(L;\Zb)$, such that  $H^1(L;\Zb)=0$ follows from the universal coefficient theorem.
In particular, every class $\beta\in H_2(X,L;\Zb)$ has vanishing Maslov index $\mu(\beta)=0$ and as $H^1(L;\Rb)=0$, it follows that $L$ is classically rigid as a special Lagrangian \cite{mclean_deformations_1998}.
If $L$ is homologically trivial, there exists a 4-chain $[\Gamma]\in H_4(X,L)$ with $\del[\Gamma]=[L]$.
Whenever $L$ is not homologically trivial, we assume there is a homologically equivalent Lagrangian sphere $L_0$ in order to correct for the failure mode addressed in Remark \hyperref[remarkbubble]{4.1}.
In this case, there exists a relative 4-chain $[\Gamma]\in H_4(X,L\cup L_{0})$ with boundary $\del[\Gamma]= [L] -[L_{0}]$.
Lastly, we will always assume that $L$ remains Lagrangian under small deformations of the K\"ahler class in $\Delta^*$.
\end{assumptions}

Under Assumptions \hyperref[assumpt]{4.2}, the bounding cochain $b=0$ is a critical point leading to disk invariants without boundary constraints, which essentially corresponds to the value of $\mathfrak{m}_{-1}$.
This is in line with the Solomon-Tukachinsky axioms implying that all disk invariants with boundary constraints vanish for $\beta\neq 0$ in case of  Lagrangians with $\mu(\beta)=0$ in Calabi-Yau threefolds.
We therefore consider a version of the axioms for enhanced open Gromov-Witten invariants adapted to the particular situation of a classically rigid special Lagrangian in a Calabi-Yau manifold in which only bulk constraints are present.
Given $\alpha_i\in H^*(X)$ with $i=1,...,n$, the invariants $OGW_\beta(\alpha_1,...,\alpha_n)$ of the pair $(X,L)$ satisfy
\begin{itemize}
	\item \textit{(Degree)} $OGW_\beta(\alpha_1,...,\alpha_n)=0$ unless 
		\begin{equation}
			\dim(X)-3+2n=\sum_{i=1}^n\text{deg}(\alpha_i).
		\end{equation}
	\item \textit{(Unit)} 
		\begin{equation}
			OGW_\beta(1,...,\alpha_{n-1})=
			\begin{cases}
				\int_\Gamma\alpha_1 & [L]=0,\, \beta=0 \text{ and } n=2\\
				0 & \text{otherwise.}
			\end{cases}
		\end{equation}
	\item \textit{(Zero)} 
		\begin{equation}
			OGW_{0}(\alpha_1,...,\alpha_{n})=
			\begin{cases}
				\int_\Gamma\alpha_1\cup \alpha_2 & [L]=0 \text{ and }n=2\\
				0 & \text{otherwise.}
			\end{cases}
		\end{equation}
	\item \textit{(Divisor)} For $\text{deg}(\alpha_n)=2$ it is
		\begin{equation}
			OGW_\beta(\alpha_1,...,\alpha_n)=\left(\int_\beta \alpha_n\right)\cdot OGW_\beta(\alpha_1,...,\alpha_{n-1}).
		\end{equation}
	\item \textit{(Invariance)} The numbers $OGW_\beta$ are constant under deformations of the K\"ahler class $\omega$ for which $L$ remains Lagrangian.
\end{itemize}

We consider the generating function of open Gromov-Witten invariants
\begin{equation}\label{opengeneratingfunction}
	\Psi=\frac{1}{(2\pi i)^2}\sum_{\beta\in H_2(X,L;\Zb)}OGW_\beta\,q^\beta=\frac{1}{2}\int_\Gamma \omega\cup \omega +\Psi_h=\frac{\lambda}{2}t^2+\Psi_h,
\end{equation}
where $\lambda=\int_\Gamma H\cup H$ and, in analogy to \eqref{GWpot}, define the quantum part 
\begin{equation}\label{OGWpot}
	\Psi_h=\frac{1}{(2\pi i)^2}\sum_{\beta\in H_2(X,L;\Zb)\setminus\{0\}}OGW_\beta\,q^\beta,
\end{equation}
which we think of as a holomorphic function on $\Delta^*$.
The function $\Psi$ can be interpreted as the BPS domain wall tension, with quantum corrections by disk instantons, between two open string vacua determined by $L$ and $L_0$.
Note that by the exact sequence
\begin{equation}\label{canonicallift}
	\cdots \longrightarrow H^\text{1}(L)\longrightarrow  H^\text{2}(X,L)\longrightarrow H^\text{2}(X)\longrightarrow H^\text{2}(L)\longrightarrow \cdots 
\end{equation}
and under Assumptions \hyperref[assumpt]{4.2}, there always exists a canonical lift of the Kähler class to relative cohomology such that the first term of \eqref{opengeneratingfunction} is well defined.\\

To characterize the Lagrangian submanifold as an A-brane it is necessary to also specify a flat $U(1)$-bundle $\mathcal{E}$ on $L$, leading to a further classical contribution to the superpotential of the pair $\mathcal{L}=(L,\mathcal{E})$ that arises from different choices of flat bundle $\mathcal{E}$ and $\mathcal{E}_{0}$ on the same Lagrangian $L$.
In addition, the $A_\infty$ structure maps $\mathfrak{m}_k$ have to be weighted by the holonomy $\text{hol}_{\del\beta}(\mathcal{E})$ of the corresponding flat connection \cite{fukaya_floer_2001} which, under Assumptions \hyperref[assumpt]{4.2}, changes the holomorphic part of the superpotential by a $r$th root of unity as a constant prefactor.
The flat bundles are topologically classified by their first Chern class in the torsion group $H^2(L;\Zb)$ and the analog of homological equivalence is the fact that they cannot be distinguished after their embedding into the simply connected ambient space $X$, i.e. 
\begin{equation}
	i_!:H^2(L;\Zb)\longrightarrow H^5(X;\Zb),\qquad  i_!\big(c_1(\mathcal{E})-c_1(\mathcal{E}_0)\big)= 0.
\end{equation}
Writing $s[\gamma]=[c_1(\mathcal{E})]-[c_1(\mathcal{E}_0)]\in H^2(L;\Zb)$, the additional contribution to the domain wall tension comes from a disk $D$ with boundary Poincar\'e dual to $s[\gamma]$ and relative homology class $s[D]\in H_2(X,L;\Zb)$. 
From the long exact sequences
\begin{equation}
\xymatrix{
	\cdots \ar[r] & H_2(X;\Zb) \ar[r]\ar[d]^-{PD} &  H_2(X,L;\Zb) \ar[r]\ar[d]^-{PD} &  H_1(L;\Zb) \ar[r]\ar[d]^-{PD} & 0\\
	\cdots \ar[r] & H^4(X;\Zb) \ar[r] &  H^4(X\setminus L;\Zb) \ar[r] & H^2(L;\Zb)\ar[r] & 0
	}
\end{equation}
we see that the class $r[D]\in H_2(X,L;\Zb)$ can be lifted to $[\ell]\in H_2(X;\Zb)$ such that the associated currents satisfy $[\delta_\ell]=r[\delta_{D}]\in H^4(X\setminus L;\Zb)$.
The contribution to the domain wall tension is then given by
\begin{equation}\label{2chain}
	\int_D\omega =\frac{s}{r}\int_{\ell}\omega =\frac{s}{r}t,\qquad s\in \Zb,
\end{equation}
and well defined due to the exact sequence \eqref{canonicallift}.\\

Under monodromy $t\mapsto t+1$ around the puncture, the domain wall tension is not invariant but changes up to integral periods that can be physically interpreted as a change in the integrally quantized flux of Ramond-Ramond fields which naturally couple to the A-brane in addition to the $U(1)$-gauge field \cite{walcher_opening_2007}.
This coupling leads to a further constant term $c$ that arises as three-loop contribution to the tension, which in full is up to closed string periods given by 
\begin{equation}\label{Amodelsuperpot}
	\mathcal{W}_A(q)=\frac{1}{2}\int_\Gamma \omega\cup \omega + \int_D\omega + c + \Psi_h.
\end{equation}
\vspace{-0.5cm}
\begin{figure}[H]
	\centering
  \includegraphics[width=0.75\textwidth]{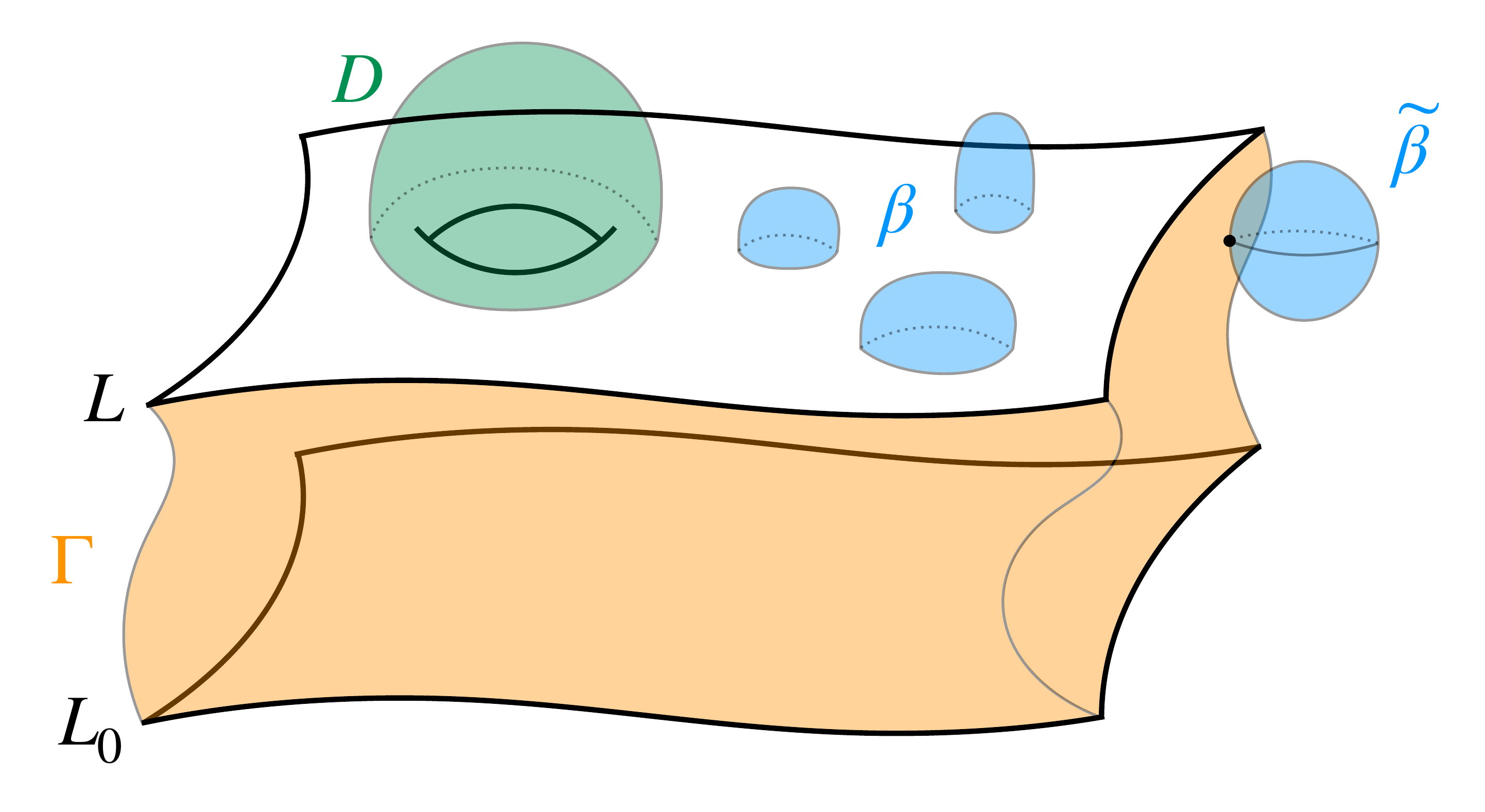}
\end{figure}
When there is a 2-chain contribution coming from a discrete open string modulus \eqref{2chain}, large volume monodromies are in general not integral.
This anomaly can be compensated by introducing several A-branes whose superpotentials are exchanged when $t\mapsto t+1$, such that the full domain wall spectrum is invariant under integral monodromy.
This is the A-model analog of an algebraic cycle branching when encircling the puncture in the B-model \eqref{Bweightfiltration} and forces the superpotential to take the general form \eqref{torsionsuperpotential}.

\subsection{Extended Mirror Pairs}

Starting from a mirror pair with isomorphic variations of MHS $\mathcal{H}^{\text{even}}(X)\cong \mathcal{H}^3(Y)$, the extension in the B-model can be given an A-model interpretation by composing the B-brane superpotential $\mathcal{W}_B$ with the mirror map.
As described in the previous section, the resulting function can be understood as A-brane superpotential $\mathcal{W}_A$ of an associated A-model geometry and defines an extension
\begin{equation}
	0\longrightarrow \mathcal{H}^{\text{even}}\longrightarrow \hat{\mathcal{H}}^{\text{even}}\longrightarrow \mathcal{I}\longrightarrow  0,\qquad \hat{\mathcal{H}}^{\text{even}}\in \text{Ext}^1_{\text{VMHS}}\left(\mathcal{I},\mathcal{H}^{\text{even}}\right)
\end{equation}
of the pure VHS $\mathcal{I}$ of weight 4 by the A-model VMHS $\mathcal{H}^{\text{even}}$.
This identification motivates the following definition in analogy to the absolute case of Section \hyperref[mirrorpairs]{2.3}.

\begin{definition}\label{extendedmirrorpair}
Let $(X,\omega)$ and $\mathcal{Y}$ be a mirror pair of Calabi-Yau threefold families, together with a collection of algebraic cycles $\mathcal{C}\subset \mathcal{Y}$ and A-branes $\mathcal{L}$ in  $X$.
We call $(X,\mathcal{L})$ and $(\mathcal{Y},\mathcal{C})$ an \textit{extended (Hodge-theoretic) mirror pair}, if the associated extensions of variations of MHS $\hat{\mathcal{H}}^{\text{even}}(X)$ and $\hat{\mathcal{H}}^3(\mathcal{Y})$ are isomorphic over small neighborhoods around MUM-type boundary points related by the mirror map.
\end{definition}

From the discussion in Section \hyperref[AJasymptotics]{3.2}, Definition \hyperref[extendedmirrorpair]{4.3} is equivalent to the identifications of the A- and B-brane superpotentials under the mirror map
\begin{equation}
	\mathcal{W}_A(q)=\mathcal{W}_B(z(q)),
\end{equation}
as the truncated normal functions $\mathcal{W}_{A/B}=Q(\widetilde{\nu},e_3)$ with respect to the canonical lift \eqref{canonical} fully determine the extension in a situation of maximal degeneracy.
Recall that we canonically normalize $e_3$ in the B-model, such that $\varpi_0(z)=1$.
Given an extended mirror pair with $\mathcal{I}$ of rank 1, the superpotential $\mathcal{W}_B$ can therefore be understood as computing the quantum corrected volume of the chain $\Gamma$ interpolating between the A-brane vacua (cf. \eqref{correctedvolume})
\begin{equation}
	\mathcal{W}=\frac{\lambda}{2}t^2+\frac{s}{r}t+c+\xi\sum_{d=1}^\infty \widetilde{n}_d\,q^{d/r}=Q(h_k,e_3),\qquad \xi^r=1,
\end{equation}
with asymptotic behavior specified by the limiting period matrix
\begin{equation}
	\hat{\Pi}_{\,q=0}=
\begin{pmatrix}[1.3]
	1 & 0 & 0 & 0 & 0 \\ 
	0 & 1 & 0 & 0 & 0 \\
	\frac{a}{24} & - \frac{\kappa}{2} & 1 & 0 & \frac{s}{r}  \\
	\frac{b\,\zeta(3)}{(2\pi i)^3} & \frac{a}{24} & 0 & 1 & c \\
	0 & 0 & 0 & 0 & 0  \\
	0 & 0 & 0 & 0 & 0 
\end{pmatrix}.
\end{equation}
In particular, the perturbative three-loop correction to the A-brane superpotential is related to the Abel-Jacobi limit in the B-model.
The identification of flat connections implies
\begin{equation}\label{twopointfunction}
	\mathfrak{D}(q(z))=\lambda+\xi\sum_{d=1}^\infty \frac{d^2\widetilde{n}_d}{r^2}\,q^{d/r}=\Psi^{\prime\prime}(q),
\end{equation}
providing an A-model interpretation of Griffiths' infinitesimal invariant \eqref{infinitesimalinvariant} as the disk two-point function of its mirror Lagrangian.
By the Solomon-Tukachinsky axioms, this corresponds to the generating function of open Gromov-Witten invariants with two bulk insertions for $\widetilde{n}_d=OGW_\beta$ and $\beta=d[\ell]$.
When $L$ arises as the fixed point locus of an anti-holomorphic involution, it is conjectured that the mirror algebraic cycle is generally defined over $\Qb$, such that the open Gromov-Witten invariants $\widetilde{n}_d\in \Qb$ are rational numbers.
However, as discussed in Section \hyperref[AJasymptotics]{3.2}, there are examples in which the B-model predicts the invariants to lie in an algebraic number field $\widetilde{n}_d\in K/\Qb$.
The analog of the multicover expansion \eqref{multicoverformula} corresponds to the characterization of $\Psi_h$ as a 2-function in the sense of \cite{schwarz_framing_2015}, see also \cite{muller_rational_2021}.
An enumerative interpretation of these invariants remains to be found.

\subsection{An A-model Origin of the Extension Class}

We consider a collection of Lagrangian submanifolds $i:L=\bigcup_kL_k\hookrightarrow X$ with $k=0,...,d$  that satisfy Assumptions \hyperref[assumpt]{4.2} and are homologically equivalent as cycles in $X$, meaning that $[L_k]-[L_{k^\prime}]=0\in H_3(X)$ for each pair of indices $k$ and $k^\prime$.
The long exact homology sequence of the pair $(X,L)$
\begin{equation}\label{longexactsequence}
	\cdots \longrightarrow H_{2k}(L)\longrightarrow H_{2k}(X)\longrightarrow H_{2k}(X,L)\longrightarrow H_{2k-1}(L)\longrightarrow \cdots 
\end{equation}
produces an extension of the even cohomology in pure weight 4 (cf. \cite{doran_algebraic_2014})
\begin{equation}\label{Amodelsequence}
\xymatrix{
	0 \ar[r] & H_4(X) \ar[r]\ar[d]^-{PD} & H_4(X,L) \ar[r]\ar[d]^-{PD} & \text{ker}\big(\,H_3(L)\overset{i_*}{\longrightarrow} H_3(X)\,\big) \ar[r]\ar[d]^-{PD} & 0\\
	0 \ar[r] & H^{2}(X) \ar[r] & H^{2}(X\setminus L) \ar[r] & \text{ker}\big(\,H^0(L)\overset{i_!}\longrightarrow H^3(X)\,\big) \ar[r] & 0
	}
\end{equation}
where the kernel of $i_!:H^0(L)\rightarrow H^3(X)$ is of rank $d$ with generators Poincar\'e dual to $[L_k]-[L_{0}]\in H_3(X)$.
Their lifts to $H^{2}(X\setminus L)$ are the currents $[\delta_{\Gamma_k}]\in H^2(X\setminus L)\cong H_4(X,L)$ defined via 
\begin{equation}\label{currents}
	\int_X \delta_{\Gamma_k}\cup \eta =\int_{\Gamma_k} \eta ,\qquad \eta\in H^4(X,L), 
\end{equation}
where $[\Gamma_k]\in H_4(X,L)$ is a four chain with boundary $\del [\Gamma_k ]=[L_k]-[L_{0}]$.
From this, we can define an extension of $\mathcal{H}^{\text{even}}$ which on the level of graded vector bundles takes the form
\begin{equation}
\hat{H}^{\text{even}}:=H^{\text{even}}(X)\oplus \Cb(-2)^{\oplus d}\langle\{\delta_{\Gamma_1},\ldots, \delta_{\Gamma_d}\}\rangle,\quad \hat{\mathcal{H}}^{\text{even}}:=\hat{H}^{\text{even}}\otimes \mathcal{O}_{\Delta^*}.
\end{equation}
An extenstion of the A-model Hodge filtration to $\hat{\mathcal{H}}^{\text{even}}$ is given by
\begin{equation}
\hat{F}^p=\bigoplus_{i\leq 3-p}\hat{H}^{2i}\subset \hat{H}^{\text{even}},\qquad \hat{\mathcal{F}}^p:=\hat{F}^p\otimes \mathcal{O}_{\Delta^*},
\end{equation}
which is seen to satisfy the conditions 
\begin{equation}
	\faktor{\hat{F}^2}{F^2}=\Cb(-2)^{\oplus d},\qquad \faktor{\hat{F}^p}{\hat{F}^{p+1}}= \faktor{F^p}{F^{p+1}}.
\end{equation}
We think of the currents \eqref{currents} as $\hat{F}^2$-lifts $f_k=[\delta_{\Gamma_k}]\in \hat{F}^2$ compatible with the Hodge filtration on $\hat{H}^{\text{even}}$.\\

The $H^2$-module structure on $H^{\text{even}}$ induced by the cup product naturally extends to $\hat{H}^\text{even}$ via
\begin{equation}\label{extendedcupproduct}
	e_2\cup f_k=[H]\cup [\delta_{\Gamma_k}]=\left(\int_X \delta_{\Gamma_k} \cup H\cup H\right)[\ell]=\left(\int_{\Gamma_k}H\cup H\right)[\ell]
\end{equation}
and we seek a deformation of this module defined in terms of the open Gromov-Witten potential which matches the extension in the B-model.
The flat connection on $\hat{\mathcal{H}}^{\text{even}}$ is therefore given by a family of algebraic structures $[H]\circledast(-): \hat{H}^{\text{even}}\rightarrow \hat{H}^{\text{even}}$ that is an augmentation of the small quantum product by holomorphic disks.
On $H^{\text{even}}$ it reduces to $e_2*(-)$, while its action on the extending generators is defined by the rule  
\begin{equation}\label{extproduct}
	\nabla_t(f_k)=e_2\circledast f_k := \sum_{l,m}\sum_{\beta\in H_2(X,L;\Zb)}OGW^k_\beta(e_2,e_l)\,q^\beta\,Q^{lm}\,e_m
	= \Psi_k^{\prime\prime}\,[\ell],
\end{equation}
where we denote by $\Psi_k$ and $OGW_\beta^k$ the superpotential and open Gromov-Witten invariants of $L_k$.
Similar to the absolute case, the energy zero contribution to this extended quantum product corresponds to the cup product $e_2\circledast f_k|_{q=0}=e_2\cup f_k$.
The extension of the Dubrovin connection is determined by both the open and closed Gromov-Witten potentials and satisfies Griffiths transversality due to $\nabla_t(f_k)\in \hat{\mathcal{F}}^1$. 
Its flatness in the one-parameter case is an immediate consequence of the algebraic properties of the operation $[H]\circledast (-)$.
In the basis $\{e_j,f_k\}$ it is represented by the matrix
\begin{equation}\label{extendeddubrovin}
\nabla_t =d+\Big([H]\circledast\Big)\otimes dt\\
:= d + 
\begin{pmatrix}
	0 & 0 & 0 & 0 & 0 & \cdots & 0\\
	1 & 0 & 0 & 0 & 0 & \cdots & 0\\
	0 & -\Phi^{\prime\prime\prime} & 0 & 0 & -\Psi^{\prime\prime}_1 & \cdots & -\Psi^{\prime\prime}_d  \\
	0 & 0 & -1 & 0 & 0 & \cdots & 0\\
	0 & 0 & 0 & 0 & 0 & \cdots & 0\\
	\vdots & \vdots & \vdots & \vdots & \vdots & \ddots & \vdots \\
 0 & 0 & 0 & 0 & 0 & \cdots & 0
\end{pmatrix}
\otimes dt,
\end{equation}
with $\nabla_t f_k=-\Psi^{\prime\prime}_k\,e_1$, c.f. \eqref{derivativeextgenerator}.
In analogy to the absolute case we construct a $\nabla$-flat $\Cb$-local system $\hat{\mathcal{H}}^{\text{even}}_\Cb$ by first setting
\begin{equation}
	\widetilde{\sigma}(f_k):=f_k+\Psi_{k,h}^\prime\, e_1+\Psi_{k,h}\, e_0
\end{equation}
in terms of the holomorphic part $\Psi_{k,h}$ of the generating function of open Gromov-Witten invariants on $L_k$, and then defining the quantum deformed current
\begin{equation}
	h_k=\sigma(f_k):=\widetilde{\sigma}(e^{-\omega}\cup f_k)=e^{-\omega}\delta_{\Gamma_k}+\Psi_{k,h}^\prime\, e_1+\Psi_{k,h}\, e_0.
\end{equation}
The local system is extended to $\hat{\mathcal{H}}^{\text{even}}_\Cb$ by $\nabla(h_k)=0$ and we regard $\hat{\mathcal{H}}^{\text{even}}_\Cb$ as a solution to the quantum differential equation coming from $[H]\circledast(-)$.\\

In the multiparameter setup of Remark \hyperref[WDVV]{2.2}, vanishing of the A-model curvature poses a further algebraic constraint on the extended quantum product coming from
\begin{equation}
R_\nabla(\del_i,\del_j)\big([\delta_{\Gamma_k}]\big)=[H_i]\circledast\big([H_j]\circledast[\delta_{\Gamma_k}]\big)-[H_j]\circledast\big([H_i]\circledast[\delta_{\Gamma_k}]\big).
\end{equation}
It can be interpreted as an equality of open Gromov-Witten invariants coming from certain boundary strata of the disk moduli space that correspond to disks with boundary on $L_k$, three marked points in the interior and no marked points on the boundary.
In terms of the Gromov-Witten potential and superpotential associated to $L_k$, this can be expressed by the following condition.

\begin{proposition}
	Flatness of the extended Dubrovin connection \eqref{extendeddubrovin} is equivalent to the concurrent validity of both the ordinary WDVV equations \eqref{WDVVequations}, and the system of partial differential equations 
\begin{equation}
	\sum_{a,b}\del_a\del_i\del_j\,\Phi\cdot Q^{ab}\cdot \del_b\del_l\,\Psi^k=\sum_{a,b}\del_a\del_i\,\Psi^k \cdot Q^{ab}\cdot \del_b\del_j\del_l\,\Phi ,\quad \text{for all } i,j,k,l,
\end{equation}
\vspace{-0.5cm}
\begin{figure}[H]
	\centering
  \includegraphics[width=0.75\textwidth]{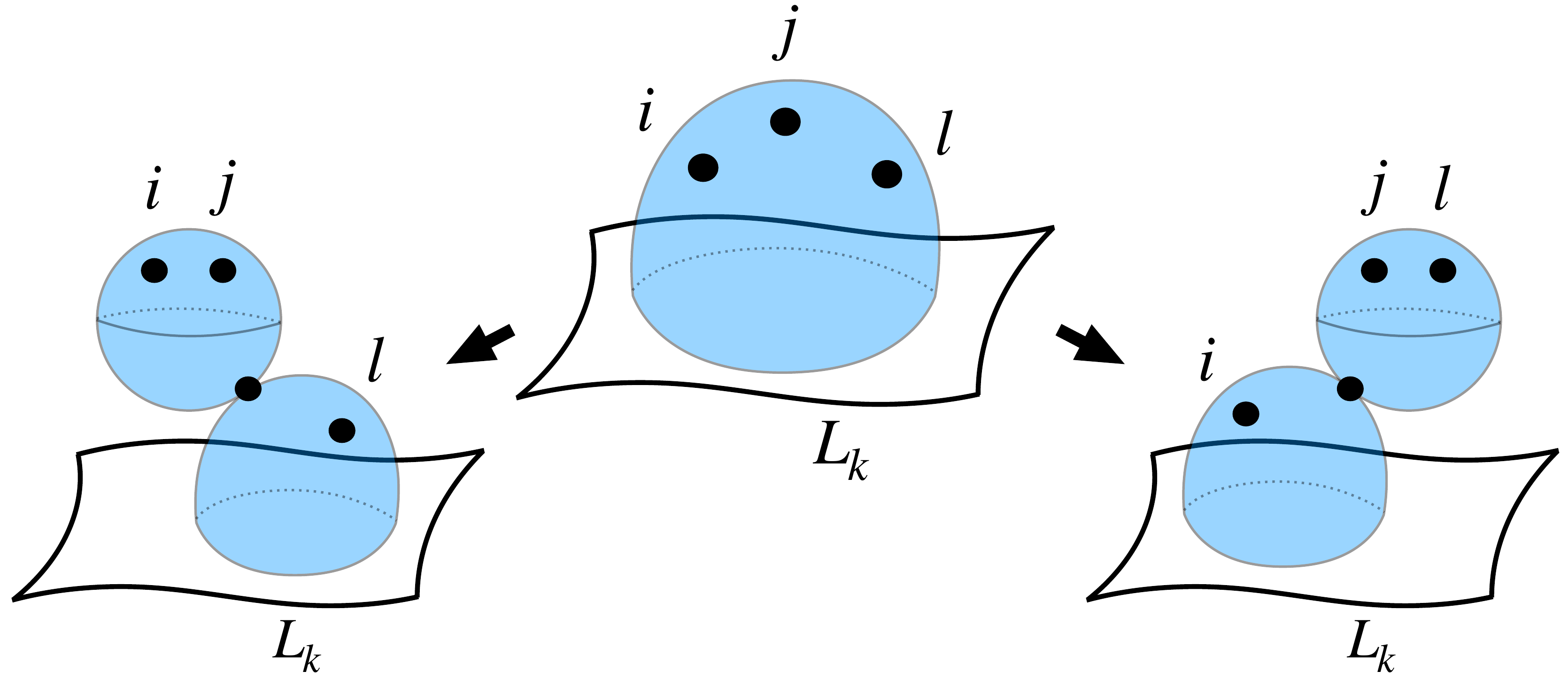}
\end{figure}
which we identify as the Open WDVV equations \cite{solomon_relative_2021,alcolado_extended_2017} under the Assumptions \hyperref[assumpt]{4.2}.
\end{proposition}

\begin{remark}
In \cite{solomon_relative_2021}, the Open WDVV equations result in the associativity of a much more general relative quantum product $\mem$ defined on a version of relative cohomology, which differs from the construction proposed here, but contains similar information in the given situation.
It would be interesting to clarify a possible correspondence, in particular in view of including boundary punctures into the present discussion.
\end{remark}

The difference of lifts $h_k-f_k$ is $d$-closed and in each fiber well defined modulo $\mathcal{F}^2+\mathcal{H}^{\text{even}}_\Zb$, where $\mathcal{H}^{\text{even}}_\Zb$ denotes Iritani's $\hat{\Gamma}$-integral local system.
The extension is then classified by a normal function $\nu_{L_k}=[h_k-f_k]\in \mathcal{J}^2(\mathcal{H}^{\text{even}})$ given by a section of the intermediate Jacobian fibration
\begin{equation}
	\mathcal{J}^2(\mathcal{H}^{\text{even}})=\frac{\mathcal{H}^\text{even}}{\mathcal{F}^2+\mathcal{H}^\text{even}_\Zb}\cong \faktor{\left(\mathcal{F}^2\right)^\vee}{\left(\mathcal{H}^\text{even}_\Zb\right)^\vee}.
\end{equation}
This normal function gives a functional defined up to integral periods that acts on each fiber by 
\begin{equation}
	\nu_{L_k}(\eta)_t:=\int_X\left(h_k-f_k\right)\cup \eta=\int_X h_k \cup \eta,\qquad \eta\in F^2H^{\text{even}}(X),
\end{equation}
where the second term vanishes for dimensionality reasons.
The truncated normal functions are given by
\begin{equation}
\mathcal{V}_k = Q(\nu_{L_k},e_2)=\int_{\Gamma_k}\omega\cup H+\Psi^\prime_{h,k},\quad 
	\mathcal{W}_k = Q(\nu_{L_k},e_3)=\frac{1}{2}\int_{\Gamma_k}\omega\cup \omega+\Psi_{h,k} 
\end{equation}
and reproduce the expected structure of the A-brane superpotential with $\mathcal{V}_k=\mathcal{W}_k^\prime$. 
As discussed around \eqref{twopointfunction}, the infinitesimal invariants correspond to the two-point functions on the disk 
\begin{equation}
	\Psi_k^{\prime\prime}=\sum_{\beta\in H_2(X,L;\Zb)}\text{OGW}^k_{\beta}(H,H)\,q^\beta=\int_X H \circledast H \circledast \delta_{\Gamma_k}=Q(\nabla_t^2\widetilde{\nu}_{L_k},e_3)=\mathfrak{D}_k.
\end{equation} 

\begin{remark}
This approach has two deficiencies: 
The additional classical contributions to the A-brane superpotential \eqref{Amodelsuperpot} are related to the choice of a specific integral local system $\hat{\mathcal{H}}^{\text{even}}_\Zb$ underlying the extension.
Analogous to the absolute case, they are expected to arise in this context as mirrors of the Abel-Jacobi limit and determine the LMHS of the extension as a constant of integration of the quantum differential equation coming from $[H]\circledast (-)$.
Their quantum interpretation in pure A-model terms presumably requires an extension of Iritani's $\hat{\Gamma}$-construction which takes the bundle data into account and correctly specifies the asymptotics of A-model periods including the constant term $Q(\widetilde{h}(0),e_3)=c$ of \eqref{specialvalue} associated to the number field that underlies the mirror object.
We conjecture that a systematic construction of an extended Gamma class and the corresponding quantum Abel-Jacobi map
\begin{equation}
	AJ:\mathcal{L}\longmapsto \nu_{\mathcal{L}}=[h-f]\in \mathcal{J}^2(\mathcal{H}^{\text{even}})
\end{equation} 
will generally reproduce the limiting periods of the mirror B-model.
Secondly, as already noted in \cite{doran_algebraic_2014}, the short exact sequence \eqref{Amodelsequence} cannot explain an extension class between two flat bundles on the same Lagrangian.
We view this as a symptom of the cohomology not fully representing the A-brane geometry and resort to the fact that the corresponding truncated normal function can often be derived after including an auxiliary reference (see Example \hyperref[realquinticdeligne]{4.8}).
\end{remark}

To make the monodromy $\hat{M}:\hat{\mathcal{H}}^{\text{even}}_\Zb\rightarrow \hat{\mathcal{H}}^{\text{even}}_\Zb$ around $q=0$ unipotent, we pass to a $r$-fold cover $\Delta_{\hat{q}}^*$ with local coordinate $\hat{q}=q^{1/r}$ and extended monodromy logarithm $\hat{N}= \log(\hat{M}^r)$.
At the puncture, $\hat{N}$ is given by $-2\pi i\,\text{Res}_{\hat{q}=0}(\nabla)$ and we consider
\begin{equation}
	\nabla_{\hat{q}}\left([\delta_{\Gamma_k}]\right)=\frac{r}{2\pi i q}\sum_{\beta\in H_2(X,L;\Zb)}\text{OGW}_{\beta}(H,H)\,q^\beta\, [\ell].
\end{equation}
As whenever $\beta\neq 0$ all holomorphic disks have positive area $\int_\beta\omega>0$, we have $q^\beta\rightarrow 0$ as $q\rightarrow 0$.
The residue is then given by
\begin{equation}
	\text{Res}_{\hat{q}=0}\nabla_{\hat{q}}\left([\delta_{\Gamma_k}]\right)=\frac{r}{2\pi i}\,\text{OGW}_{0}(H,H)\,[\ell]=\frac{r}{2\pi i}\left(\int_{\Gamma_k} H\cup H\right)[\ell],
\end{equation}
which is $r/(2\pi i)$ times the only non-vanishing entry of the matrix representing the cup product $[H]\cup [\delta_{\Gamma_k}]$.
Along similar lines \cite{cox_mirror_1999} we can compute the remaining entries of the residue matrix such that the monodromy logarithm $\hat{N}= \log(\hat{M}^r)$ of the extended A-model VMHS at $0\in \Delta_{\hat{q}}$ is given by $-r[H]\cup (-):\hat{H}^{\text{even}}\rightarrow \hat{H}^{\text{even}}$, which in the basis $\{e_i,f_k\}$ is represented by the matrix
\begin{equation}
 \hat{N}_{\hat{q}=0}
 =
 	\begin{pmatrix}
	0 & 0 & 0 & 0 & 0 & \cdots & 0\\
	-r & 0 & 0 & 0 & 0 & \cdots & 0\\
	0 & r\kappa & 0 & 0 & \lambda_1 & \cdots & \lambda_d \\
	0 & 0 & r & 0 & 0 & \cdots & 0\\
	0 & 0 & 0 & 0 & 0 & \cdots & 0\\
	\vdots & \vdots & \vdots & \vdots & \vdots & \ddots & \vdots \\
 0 & 0 & 0 & 0 & 0 & \cdots & 0
\end{pmatrix}.
\end{equation}
Here, $r$ corresponds to the order of the torsion group $H^2(L;\Zb)$.
The associated relative weight filtration $\hat{W}_\bullet=\hat{W}_\bullet(\hat{N})$ leads to a filtration $\hat{\mathcal{W}}_\bullet=\hat{W}\otimes\mathcal{O}_{{\Delta}^*}$ of $\hat{\mathcal{H}}^{\text{even}}$ such that the graded pieces fit into the diagram with horizontal exact sequences
\begin{equation}\label{Aweightfiltration}
	\xymatrix{
	& \mathcal{H}^0 \ar[d]_-{rN} \ar[r]^-{\sim} & \hat{\mathcal{H}}^0 \ar[d]^-{\hat{N}} & & \\
	0 \ar[r] & \mathcal{H}^2 \ar[d]_-{rN} \ar[r] & \hat{\mathcal{H}}^2 \ar[d]^-{\hat{N}} \ar[r] & \mathcal{I}_\Zb \ar[r]\ar[d]^-{L} & 0\\
	0 \ar[r] & \mathcal{H}^4 \ar[d]_-{rN} \ar[r] & \hat{\mathcal{H}}^4 \ar[d]^-{\hat{N}} \ar[r] &    H   \ar[r] & 0\\
	& \mathcal{H}^6 \ar[r]^-{\sim} & \hat{\mathcal{H}}^6, & &
	}
\end{equation}
cf. \eqref{longexactsequence}. There is an immediate resemblance with \eqref{Bweightfiltration}, where we in particular identify the torsion group $H$ in \eqref{Bweightfiltration} with $H^2(L;\Zb)$ and the degree of the branched covering $\hat{\mathcal{Y}}\rightarrow \Delta^*_{\hat{z}}$ with its order.

\begin{remark}
The nilpotent orbit associated to the extension $\hat{\mathcal{H}}^\text{even}$ is fully determined by the $H^2$-module structure on $\hat{H}^\text{even}$ defined by the cup product \eqref{extendedcupproduct}.
Its deformation to an extended quantum product subject to the Open WDVV equations can be viewed as an extended version of a Frobenius module in the sense of \cite{cattani_frobenius_2003}.
\end{remark}

We now consider the main example of \cite{walcher_opening_2007}, concerning a mirror pair of cycles lying in the Fermat quintic and mirror quintic family discussed in Example \hyperref[closedexample]{2.3}.

\begin{example}\label{realquinticdeligne}
On the A-side, an anti-holomorphic involution $\iota:\Pb^4\rightarrow \Pb^4$ with $\iota(x_i)=\overline{x}_i$ restricted to the Fermat quintic \eqref{fermatquintic} has a fixed point locus
\begin{equation}\label{reallocus}
	L=\left\{y_1^5+y_2^5+y_3^5+y_4^5+y_5^5=0\right\},\qquad y_i\in \Rb,
\end{equation}
which is special Lagrangian with respect to the Kähler class $\omega$ and holomorphic 3-form $\Omega$ on $X$.
Solving \eqref{reallocus} for any $y_i$ over $\Rb$ topologically identifies this real quintic with a copy of real projective space $L\cong \Rb\Pb^3$.
Flat U(1)-bundles $\mathcal{E}$ on $L$ are up to isomorphism classified by their first Chern class $c_1(\mathcal{E})\in H^2(L;\Zb)\cong \Zb/2\Zb$ such that there are two A-brane configurations which we denote by $\mathcal{L}_+=(L,\mathcal{E}_+)$ and $\mathcal{L}_-=(L,\mathcal{E}_-)$.
Based on constraints imposed by consistency of A-brane charges \cite{walcher_opening_2007}, the most general superpotentials associated to $\mathcal{L}_{\pm}$ are up to closed string periods, i.e. $\text{mod}\enskip t\Zb+\Zb$, given by 
\begin{equation}\label{twovacua}
\mathcal{W}_+=\frac{t^2}{4} +\Psi_h,\qquad \mathcal{W}_-=\frac{t^2}{4}-\frac{t}{2}+\frac{1}{4} - \Psi_h, \qquad\Psi_h=\frac{1}{(2\pi i)^2}\sum_{d=1}^\infty \widetilde{n}_d\, q^{d/2},
\end{equation}
with $\widetilde{n}_d=OGW_\beta$ for $\beta=d\,[\ell]$.
We interpret these superpotentials as the domain wall tensions of $\mathcal{L}_{\pm}$ in relation to a third Lagrangian $L_0\subset X$, where the classical currents \eqref{currents} account for the terms proportional to $t^2$.
Thought of the superpotentials as truncated normal functions $\mathcal{W}_\pm=Q(\widetilde{\nu}_\pm,e_3)$, the class $\nu_\pm=[\widetilde{\nu}_\pm]\in\mathcal{J}^2(\mathcal{H}^{\text{even}})$ defines an extension of the A-model VMHS associated to $(X,\omega)$.
Note that because the superpotentials are measured with respect to the same Lagrangian $L_0$ up to homology, the corresponding $\hat{\mathcal{F}}^2$-lifts $f_\pm$ coincide, while the quantum deformed currents $h_\pm$ differ due to different choices of flat bundle on $L$. 
Under monodromy $\hat{M}$ of the corresponding extended local system the extending generators $h_\pm$ are exchanged, such that we have to consider a pullback along the double cover $\hat{q}=q^{1/2}\mapsto q$ and the monodromy logarithm $\hat{N}=\log(\hat{M}^2)$.
In the basis $\{g_i,h_\pm\}$ the respective matrices are given by
\begin{equation}
\hat{M}=
\begin{pmatrix}
1 & 0 & 0 & 0 & 0 & 0\\
-1 & 1 & 0 & 0 & 0 & 0\\
0 & 5 & 1 & 0 & 1 & 0\\
-5 & 5 & 1 & 1 & 0 & 0\\
0 & 0 & 0 & 0 & 0 & 1\\
0 & 0 & 0 & 0 & 1 & 0
\end{pmatrix},\qquad
\hat{N}=
\begin{pmatrix}
0 & 0 & 0 & 0 & 0 & 0\\
-2 & 0 & 0 & 0 & 0 & 0\\
5 & 10 & 0 & 0 & 1 & 1\\
-\frac{25}{3} & 5 & 2 & 0 & 0 & -1\\
0 & 0 & 0 & 0 & 0 & 0\\
0 & 0 & 0 & 0 & 0 & 0
\end{pmatrix}.
\end{equation}
Taking the difference $\mathcal{T}_A=\mathcal{W}_+-\mathcal{W}_-$ yields the domain wall tension between the vacua specified by the bundles $\mathcal{E}_+$ and $\mathcal{E}_-$ on the same Lagrangian. It is $\text{mod}\enskip t\Zb+\Zb$ of the form
\begin{equation}\label{TA}
	\mathcal{T}_A = \frac{t}{2}-\frac{1}{4} + 2\Psi_h = \frac{t}{2}-\left(\frac{1}{4} + \frac{1}{2\pi^2}\sum_{d=0}^{\infty} \widetilde{n}_d \, q^{d/2}\right)
\end{equation}
and likewise defines an extension of $\mathcal{H}^{\text{even}}$ on the double cover with canonical lift of the normal function $\widetilde{\nu}=\mathcal{T}_A^\prime\, e_1+\mathcal{T}_A\,e_0$ and monodromy logarithm
\begin{equation}
\hat{N}=
\begin{pmatrix}
0 & 0 & 0 & 0 & 0\\
-2 & 0 & 0 & 0 & 0\\
5 & 10 & 0 & 0 & 0\\
-\frac{25}{3} & 5 & 2 & 0 & 1\\
 0 & 0 & 0 & 0 & 0
\end{pmatrix}.
\end{equation}
On the B-side we consider the image of the Deligne conics 
\begin{equation}
		C_{\pm,\psi}=\left\{x_1+x_2=0,\enskip x_3+x_4=0,\enskip x_5^2=\pm\sqrt{5\psi}\,x_1x_3=0\right\}\subset Y_\psi
\end{equation}
under the Greene-Plesser quotient \cite{morrison_d-branes_2009}.
As the components of $C_{\pm,\psi}$ are exchanged under monodromy around $z=0$, the algebraic curve $\mathcal{C}_\pm\subset \hat{\mathcal{Y}}$ is only globally well defined after passing to the double cover. 
Because $\mathcal{C}_+-\mathcal{C}_-\in \text{CH}^2_{\text{hom}}(\mathcal{Y})$, any family of 3-chains $\Gamma_z\in H_3(Y_z;\Zb)$ defines a normal function with superpotential
\begin{equation}
	\mathcal{T}_B=\int_{\Gamma_z}\Omega_z.
\end{equation}
\end{example}

The statement of extended mirror symmetry can now be formulated as the following Mirror Theorem for the quintic.

\begin{theorem}
	The cycles given by both choices of flat bundle on the real quintic $\mathcal{L}_\pm$ in $X$ and the Deligne conics $\mathcal{C}_\pm\subset \hat{\mathcal{Y}}$ determine an extended mirror pair in the sense of Definition \hyperref[extendedmirrorpair]{4.3}.
\end{theorem}

\begin{proof}
The statement is a consequence of the results in \cite{walcher_opening_2007,pandharipande_disk_2008,morrison_d-branes_2009}.
The B-brane superpotential $\mathcal{T}_B$ satisfies the inhomogeneous Picard-Fuchs equation \cite{morrison_d-branes_2009}
\begin{equation}
	D_{\text{PF}}\,(-)=(2\pi i)^2\left[\theta^4-5z\prod_{k=1}^4\left(5\theta+k\right)\right](-)=-\frac{15}{4}\sqrt{z},
\end{equation}
for the standard choice of holomorphic 3-form \eqref{standardform} on $\mathcal{Y}$ and its particular solution $\tau(z)$ computes the open Gromov-Witten potential of the real quintic $X$.
By the localization calculation of \cite{pandharipande_disk_2008} it is obtained by canonical normalization and composition with the mirror map
\begin{equation}
	\tau(z)=2\sum_{d\text{ odd}}\frac{(5d)!!}{(d!!)^5}\,z^{d/2},\qquad (2\pi i )^2\,\Psi_h(q)=\frac{\tau(z(q))}{\varpi_0(z(q))}.
\end{equation}
This identification establishes an equivalence of the A- and B-brane superpotentials up to solutions of the homogeneous Picard-Fuchs equation.
Employed as truncated normal functions in terms of the canonical lift \eqref{canonical}, this induces an isomorphism of the underlying complex variations of MHS.
Equality of integral local systems corresponds to an identification of the classical terms in the superpotentials which determine the LMHS.
Analytic continuation of $\mathcal{T}_B$ to the Gepner point provides the LMHS of the B-model \cite{walcher_opening_2007}
\begin{equation}
	\mathcal{T}_B(z)=\frac{\varpi_1(z)}{2}-\left(\frac{\varpi_0(z)}{4}+\frac{1}{2\pi^2}\,\tau(z)\right),\qquad \mathcal{T}_A(q)=\frac{\mathcal{T}_B(z(q))}{\varpi_0(z(q))},
\end{equation}
which reproduces \eqref{TA} on the A-side.
\end{proof}

The flat connection in the canonical coordinate is determined by the infinitesimal invariant 
\begin{equation}
	\mathfrak{D}(z(q))=2\Psi_h^{\prime\prime}=15\,q^{1/2}+6900\, q^{3/2}+13603140\,q^{5/2}+\cdots,
\end{equation}
which we identify as twice the two-point function on the disk by the Solomon-Tukachinsky axioms.
Here, the open Gromov-Witten potential of $L$ is given by 
\begin{equation}
	(2\pi i )^2\,\Psi_h=30\,q^{1/2}+\frac{4600}{3}\,q^{3/2}+\frac{5441256}{5}\,q^{5/2}+\cdots.
\end{equation}
The functional which mirrors the Abel-Jacobi map in the A-model is fully determined by its action on $\mathcal{F}^2$ with limiting A-model periods
\begin{equation}
	Q(\widetilde{h}(0),e_2)=\frac{1}{2},\qquad Q(\widetilde{h}(0),e_3)=-\frac{1}{4}.
\end{equation}
Here, the first component is related to the double cover and we can think of the second component 
\begin{equation}
	-\frac{1}{4}=6\,\frac{\zeta(2)}{(2\pi i)^2}\in\faktor{\Cb}{\Zb(2)}
\end{equation}
as a regulator image of the degenerated Deligne conics.

\bibliography{EAVHS.bib}

\begin{thebibliography}{CdlOGP}

\bibitem[Alc]{alcolado_extended_2017}
Adam Alcolado.
\newblock {\em Extended {Frobenius} {Manifolds} and the {Open} {WDVV}
  {Equations}}.
\newblock Ph.{D}. thesis, McGill University, 2017.

\bibitem[Car]{carlson_extensions_1979}
James~A. Carlson.
\newblock Extensions of mixed {Hodge} structures.
\newblock {\em Journées de Géometrie Algébrique d'Angers}, pages 107--127,
  1979.

\bibitem[CdlOGP]{candelas_pair_1991}
Philip Candelas, Xenia~C. de~la Ossa, Paul~S. Green, and Linda Parkes.
\newblock A pair of {Calabi}-{Yau} manifolds as an exactly soluble
  superconformal theory.
\newblock {\em Nucl. Phys. B}, 359:21--74, 1991.

\bibitem[CF1]{cattani_asymptotic_2001}
Eduardo Cattani and Javier Fernandez.
\newblock Asymptotic {Hodge} theory and quantum products.
\newblock {\em Contemp. Math.}, 276:115--136, 2001.

\bibitem[CF2]{cattani_frobenius_2003}
Eduardo Cattani and Javier Fernandez.
\newblock Frobenius {Modules} and {Hodge} {Asymptotics}.
\newblock {\em Commun. Math. Phys.}, 238:489--504, 2003.

\bibitem[CK1]{cattani_eduardo_degenerating_1989}
Eduardo Cattani and Aroldo Kaplan.
\newblock Degenerating variations of {Hodge} structure.
\newblock In {\em Théorie de {Hodge} - {Luminy}, {Juin} 1987}, number 179-180
  in Astérisque. 1989.

\bibitem[CK2]{cox_mirror_1999}
David~A. Cox and Sheldon Katz.
\newblock {\em Mirror symmetry and algebraic geometry}, volume~68 of {\em Math.
  {Surveys} and {Monographs}}.
\newblock American Mathematical Society, 1999.

\bibitem[Del]{deligne_local_1997}
Pierre~R. Deligne.
\newblock Local behavior of {Hodge} structures at infinity.
\newblock In {\em Mirror {Symmetry} {II}}, volume~1 of {\em {AMS}/{IP} {Stud}.
  {Adv}. {Math}.}, pages 683--699. 1997.

\bibitem[DK]{doran_algebraic_2014}
Charles~F. Doran and Matt Kerr.
\newblock Algebraic cycles and local quantum cohomology.
\newblock {\em Comm. Num. Th. Phys.}, 8:703--727, 2014.

\bibitem[DM]{doran_mirror_2006}
Charles~F. Doran and John~W. Morgan.
\newblock Mirror {Symmetry} and {Integral} {Variations} of {Hodge} {Structure}
  {Underlying} {One} {Parameter} {Families} of {Calabi}-{Yau} {Threefolds}.
\newblock In {\em Mirror {Symmetry} {V}}, volume~38 of {\em {AMS}/{IP} {Stud}.
  {Adv}. {Math}.}, pages 517--537. 2006.

\bibitem[dSJKP]{da_silva_jr_arithmetic_2016}
Genival da~Silva~Jr., Matt Kerr, and Gregory Pearlstein.
\newblock Arithmetic of degenerating principal variations of {Hodge} structure:
  examples arising from mirror symmetry and middle convolution.
\newblock {\em Canad. J. Math.}, 68:280--308, 2016.

\bibitem[DVV1]{dijkgraaf_notes_1991}
Robbert Dijkgraaf, Herman Verlinde, and Erik Verlinde.
\newblock Notes on {Topological} {String} {Theory} and {2D} {Quantum}
  {Gravity}.
\newblock In {\em String {Theory} and {Quantum} {Gravity}}, Proc. of the
  {Trieste} {Spring} {School} 1990, pages 91--156. 1991.

\bibitem[DVV2]{dijkgraaf_topological_1991}
Robbert Dijkgraaf, Herman Verlinde, and Erik Verlinde.
\newblock Topological strings in d{\textless}1.
\newblock {\em Nucl. Phys. B}, 352:59--86, 1991.

\bibitem[FOOO]{fukaya_lagrangian_2009}
Kenji Fukaya, Yong-Geun Oh, Hiroshi Ohta, and Kaoru Ono.
\newblock {\em Lagrangian {Intersection} {Floer} {Theory}: {Anomaly} and
  {Obstruction}}, volume~46 of {\em {AMS}/{IP} {Stud}. {Adv}. {Math}.}
\newblock American Mathematical Society, 2009.

\bibitem[Fuk1]{fukaya_floer_2001}
Kenji Fukaya.
\newblock Floer homology and mirror symmetry {I}.
\newblock {\em AMS/IP Stud. Adv. Math.}, 23:15--43, 2001.

\bibitem[Fuk2]{fukaya_counting_2011}
Kenji Fukaya.
\newblock Counting pseudo-holomorphic discs in {Calabi}-{Yau} 3-holds.
\newblock {\em Tohoku Math. J.}, 63:697--727, 2011.

\bibitem[GGK1]{green_neron_2007}
Mark Green, Phillip Griffiths, and Matt Kerr.
\newblock Néron models and boundary components for degenerations of {Hodge}
  structure of mirror quintic type.
\newblock In {\em Curves and {Abelian} {Varieties}}, volume 465 of {\em
  Contemp. {Math}.}, pages 71--145. 2007.

\bibitem[GGK2]{green_neron_2010}
Mark Green, Phillip Griffiths, and Matt Kerr.
\newblock Néron models and limits of {Abel}–{Jacobi} mappings.
\newblock {\em Compositio Math.}, 146:288--366, 2010.

\bibitem[Giv]{givental_equivariant_1996}
Alexander~B. Givental.
\newblock Equivariant {Gromov}-{Witten} {Invariants}.
\newblock {\em Internat. Math. Res. Notices}, 13:613--663, 1996.

\bibitem[GPS]{ganatra_mirror_2015}
Sheel Ganatra, Timothy Perutz, and Nick Sheridan.
\newblock Mirror symmetry: from categories to curve counts.
\newblock {\em arXiv:1510.03839}, 2015.

\bibitem[GvdVZ]{grisaru_four-loop_1986}
Marcus~T. Grisaru, Anton E.~M. van~den Ven, and Daniela Zanon.
\newblock Four-loop beta-function for the {N}=1 and {N}=2 supersymmetric
  non-linear sigma model in two dimensions.
\newblock {\em Phys. Lett. B}, 173:423--428, 1986.

\bibitem[Iri]{iritani_integral_2009}
Hiroshi Iritani.
\newblock An integral structure in quantum cohomology and mirror symmetry for
  toric orbifolds.
\newblock {\em Adv. Math.}, 222:1016--1079, 2009.

\bibitem[JW]{jefferson_monodromy_2014}
Robert~A. Jefferson and Johannes Walcher.
\newblock Monodromy of inhomogeneous {Picard}–{Fuchs} equations.
\newblock {\em Comm. Num. Th. Phys.}, 8:1--40, 2014.

\bibitem[KKP]{katzarkov_hodge_2008}
Ludmil Katzarkov, Maxim Kontsevich, and Tony Pantev.
\newblock Hodge theoretic aspects of mirror symmetry, {From} {Hodge} theory to
  integrability and {TQFT} tt*-geometry.
\newblock {\em Proc. Sympos. Pure Math.}, 78:87--174, 2008.

\bibitem[KM]{kontsevich_gromov-witten_1994}
Maxim Kontsevich and Yuri~I. Manin.
\newblock Gromov-{Witten} classes, quantum cohomology, and enumerative
  geometry.
\newblock {\em Commun. Math. Phys.}, 164:525--562, 1994.

\bibitem[Kon]{kontsevich_homological_1994}
Maxim Kontsevich.
\newblock Homological {Algebra} of {Mirror} {Symmetry}.
\newblock {\em Proc. ICM}, 1:120--139, 1994.

\bibitem[LLY]{lian_mirror_1997}
Bong~H. Lian, Kefeng Liu, and Shing-Tung Yau.
\newblock Mirror {Principle} {I}.
\newblock {\em Asian J. Math.}, 1:729--763, 1997.

\bibitem[LW]{laporte_monodromy_2012}
Guillaume Laporte and Johannes Walcher.
\newblock Monodromy of an {Inhomogeneous} {Picard}-{Fuchs} {Equation}.
\newblock {\em SIGMA}, 8:056, 2012.

\bibitem[McL]{mclean_deformations_1998}
Robert~C. McLean.
\newblock Deformations of {Calibrated} {Submanifolds}.
\newblock {\em Comm. Anal. Geom.}, 6:705--747, 1998.

\bibitem[Mor1]{morrison_mirror_1993}
David~R. Morrison.
\newblock Mirror symmetry and rational curves on quintic threefolds: a guide
  for mathematicians.
\newblock {\em J. Am. Math. Soc.}, 6:223--247, 1993.

\bibitem[Mor2]{morrison_mathematical_1997}
David~R. Morrison.
\newblock Mathematical {Aspects} of {Mirror} {Symmetry}.
\newblock In {\em Complex {Algebraic} {Geometry}}, volume~3 of {\em
  {IAS}/{Park} {City} {Math}. {Series}}, pages 265--340. 1997.

\bibitem[MW]{morrison_d-branes_2009}
David~R. Morrison and Johannes Walcher.
\newblock D-branes and normal functions.
\newblock {\em Adv. Theor. Math. Phys.}, 13:553--598, 2009.

\bibitem[Mü]{muller_rational_2021}
L.~Felipe Müller.
\newblock Rational 2-functions are abelian.
\newblock {\em Comm. Num. Th. Phys.}, 15:605--614, 2021.

\bibitem[PSW]{pandharipande_disk_2008}
Rahul Pandharipande, Jake~P. Solomon, and Johannes Walcher.
\newblock Disk enumeration on the quintic 3-fold.
\newblock {\em J. Am. Math. Soc.}, 21:1169--1209, 2008.

\bibitem[Sch]{schmid_variation_1973}
Wilfried Schmid.
\newblock Variation of hodge structure: {The} singularities of the period
  mapping.
\newblock {\em Inventiones Mathematicae}, 22:211--319, 1973.

\bibitem[She]{sheridan_homological_2015}
Nick Sheridan.
\newblock Homological {Mirror} {Symmetry} for {Calabi}-{Yau} hypersurfaces in
  projective space.
\newblock {\em Inventiones mathematicae}, 199:1--186, 2015.

\bibitem[ST1]{solomon_differential_2020}
Jake~P. Solomon and Sara~B. Tukachinsky.
\newblock Differential forms, {Fukaya} {A}-infinity algebras, and
  {Gromov}-{Witten} axioms.
\newblock {\em arXiv:1608.01304}, 2020.

\bibitem[ST2]{solomon_point-like_2021}
Jake~P. Solomon and Sara~B. Tukachinsky.
\newblock Point-like bounding chains in open {Gromov}-{Witten} theory.
\newblock {\em Geom. Funct. Anal.}, 31:1245--1320, 2021.

\bibitem[ST3]{solomon_relative_2021}
Jake~P. Solomon and Sara~B. Tukachinsky.
\newblock Relative quantum cohomology.
\newblock {\em arXiv:1906.04795}, 2021.

\bibitem[SVW1]{schwarz_framing_2015}
Albert Schwarz, Vadim Vologodsky, and Johannes Walcher.
\newblock Framing the {Di}-{Logarithm} (over {Z}).
\newblock {\em Proc. Symp. Pure Math.}, 90:113--128, 2015.

\bibitem[SVW2]{schwarz_integrality_2017}
Albert Schwarz, Vadim Vologodsky, and Johannes Walcher.
\newblock Integrality of {Framing} and {Geometric} {Origin} of 2-functions
  (with algebraic coefficients).
\newblock {\em arXiv:1702.07135}, 2017.

\bibitem[Usu]{usui_studies_2014}
Sampei Usui.
\newblock Studies of closed/open mirror symmetry for quintic threefold through
  log mixed {Hodge} theory.
\newblock {\em arXiv:1404.7687}, 2014.

\bibitem[Voi]{voisin_hodge_2002}
Claire Voisin.
\newblock {\em Hodge {Theory} and {Complex} {Algebraic} {Geometry} {I},{II}},
  volume~1 of {\em Cambridge {Stud}. {Adv}. {Math}.}
\newblock Cambridge University Press, 2002.

\bibitem[Wal]{walcher_opening_2007}
Johannes Walcher.
\newblock Opening {Mirror} {Symmetry} on the {Quintic}.
\newblock {\em Commun. Math. Phys.}, 276:671--689, 2007.

\bibitem[Wit]{witten_structure_1990}
Edward Witten.
\newblock On the structure of the topological phase of two-dimensional gravity.
\newblock {\em Nucl. Phys. B}, 340:281--332, 1990.

\end{thebibliography}
\bibliographystyle{alphanum}

\vspace{0.5cm}

\textsc{Mathematical Institute, Heidelberg University}\\
\textsc{Im Neuenheimer Feld 205, 69120 Heidelberg, Germany}\\

\textit{e-mail:}\quad\href{mailto:lhahn@mathi.uni-heidelberg.de}{lhahn@mathi.uni-heidelberg.de},\quad\href{mailto:walcher@uni-heidelberg.de}{walcher@uni-heidelberg.de}

\end{document}